\documentclass[a4paper,10pt]{amsart}
\usepackage[utf8]{inputenc}
\usepackage{amsmath}
\usepackage{amssymb}
\usepackage{amsthm}
\usepackage{tikz-cd}
\usetikzlibrary{positioning}
\usetikzlibrary{decorations.pathreplacing}
\usepackage[foot]{amsaddr}
\usepackage{hyperref}

\newtheorem{theorem}{Theorem}[section]
\newtheorem{proposition}[theorem]{Proposition}
\newtheorem{corollary}[theorem]{Corollary}
\newtheorem{lemma}[theorem]{Lemma}
\newtheorem{definition}[theorem]{Definition}
\newtheorem{remark}[theorem]{\bf Remark}

\DeclareMathOperator{\Frac}{Frac}

\DeclareMathOperator{\Ad}{Ad}
\DeclareMathOperator{\ad}{ad}
\DeclareMathOperator{\height}{ht}

\newcommand{\system}{S}
\newcommand{\diffpolyring}{R}
\DeclareMathOperator{\Sol}{Sol}

\newcommand{\N}{\mathbb{N}}
\newcommand{\Z}{\mathbb{Z}}
\newcommand{\Q}{\mathbb{Q}}
\newcommand{\C}{\mathbb{C}}

\newcommand{\gauge}[2]{#1 . #2}

\newcommand{\dlog}{\ell\delta}

\newcommand{\group}{G}
\newcommand{\GL}{\mathrm{GL}}
\newcommand{\SL}{\mathrm{SL}}
 
\newcommand{\liealg}{\mathfrak{g}}
\newcommand{\gl}{\mathfrak{gl}}
\newcommand{\liesl}{\mathfrak{sl}}

\newcommand{\field}{C}
\newcommand{\difffield}{F}
\newcommand{\extfield}{E}
\newcommand{\resfield}{\mathcal{K}}
\newcommand{\algext}{\mathcal{L}}
\newcommand{\diffring}{\mathcal{R}}
\newcommand{\locdiffring}{\mathcal{S}}
\newcommand{\multset}{\mathcal{M}}
\newcommand{\ideal}{I}
\newcommand{\otherfield}{K}
\newcommand{\juanext}{\mathcal{E}_J}
\newcommand{\generalext}{\mathcal{E}}
\newcommand{\nongenext}{L}

\newcommand{\roots}{\Phi}
\newcommand{\rootbasis}{\Delta}
\newcommand{\weyl}{\mathcal{W}}
\newcommand{\cartan}{\mathfrak{h}}
\newcommand{\torus}{T}
\newcommand{\borel}{B}
\newcommand{\unipotent}{U}

\newcommand{\ALiou}{A_{\rm Liou}}
\newcommand{\YLiou}{Y_{\rm Liou}}
\newcommand{\calYLiou}{\mathcal{Y}_{\rm Liou}}

\newcommand{\bap}{\mathbf{a}^+}
\newcommand{\bam}{\mathbf{a}^-}
\newcommand{\baz}{\mathbf{a}^0}
\newcommand{\bb}{\mathbf{b}}
\newcommand{\baml}{\mathbf{a}_{\neg \Psi}^-}

\newcommand{\bapi}{\mathbf{\overline{a}}^+}
\newcommand{\bami}{\mathbf{\overline{a}}^-}
\newcommand{\bazi}{\mathbf{\overline{a}}^0}
\newcommand{\bbi}{\mathbf{\overline{b}}}
\newcommand{\bamil}{\mathbf{\overline{a}}_{\neg \Psi}^-}

\newcommand{\api}{\overline{a}^+}
\newcommand{\ami}{\overline{a}^-}
\newcommand{\azi}{\overline{a}^0}
\newcommand{\bi}{\overline{b}}

\newcommand{\bhp}{\mathbf{\overline{h}}^+}
\newcommand{\bhm}{\mathbf{\overline{h}}^-}
\newcommand{\bhz}{\mathbf{\overline{h}}^0}

\newcommand{\bgp}{\mathbf{\overline{g}}^+}
\newcommand{\bgm}{\mathbf{\overline{g}}^-}
\newcommand{\bgz}{\mathbf{\overline{g}}^0}

\newcommand{\bt}{\mathbf{\overline{t}}}
\newcommand{\bv}{\mathbf{v}}
\newcommand{\bs}{\mathbf{s}}
\newcommand{\bc}{\mathbf{c}}
 \newcommand{\bz}{\mathbf{z}}

\newcommand{\bsap}{\mathbf{a}^+_{*}}
\newcommand{\bsam}{\mathbf{a}^-_{*}}
\newcommand{\bsaz}{\mathbf{a}^0_{*}}
\newcommand{\bsb}{\mathbf{b}_{*}}
\newcommand{\bsbi}{\mathbf{\overline{b}_{*}}}
\newcommand{\lweyl}{\overline{w}}

\newcommand{\Imw}{\ideal_{\lweyl}}

\newcommand{\ndvar}{n}

\newcommand{\maxd}{\mathrm{max}}
\newcommand{\sep}{\mathrm{sep}}
\newcommand{\ord}{\mathrm{ord}}

\title[Normal Forms for the Classical Groups]{Normal Forms in Differential Galois Theory for the Classical Groups}
\author{Daniel Robertz}
\address{Lehrstuhl f\"ur Algebra und Zahlentheorie, RWTH Aachen University, D--52056 Aachen, Germany} 
\email{daniel.robertz@rwth-aachen.de}
\author{Matthias Sei\ss}
\address{Institut für Mathematik, Universit\"at Kassel, D--34109 Kassel, Germany}
\email{mseiss@mathematik.uni-kassel.de}

\begin{document}

\maketitle
\begin{abstract}
Let $\group$ be a classical group of dimension $d$ and 
let $\boldsymbol{a}=(a_1,\dots,a_d)$ be differential indeterminates over a 
differential field $\difffield$ of characteristic zero with algebraically closed field of constants $\field$.
Further let $A(\boldsymbol{a})$ be a generic element in the Lie algebra $ \liealg(\difffield\langle \boldsymbol{a} \rangle)$
of $\group$ obtained from parametrizing a basis of $\liealg$ with the indeterminates $\boldsymbol{a}$. 
It is known  (cf.~\cite{Juan}) that the differential Galois group of $\boldsymbol{y}'=A(\boldsymbol{a})\boldsymbol{y}$ over $F\langle \boldsymbol{a} \rangle$ is $\group(\field)$.
In this paper we construct a differential field extension $\algext$ of $F\langle \boldsymbol{a} \rangle$ such that 
the field of constants of $\algext$ is $\field$, the differential Galois group of $\boldsymbol{y}'=A(\boldsymbol{a})\boldsymbol{y}$ over $\algext$ is still the full group $\group(\field)$ and $A(\boldsymbol{a})$ is gauge equivalent over $\algext$ to a matrix in normal form which we introduced in \cite{Seiss}. We also consider specializations of the coefficients of $A(\boldsymbol{a})$. 
\end{abstract}

\section{Introduction}
In classical Galois theory there is the well-known construction of the general polynomial equation over $\Q$ with Galois group the symmetric group $S_n$.
The coefficients of the general equation are (up to sign) the elementary symmetric polynomials in $n$ indeterminates $x_1$, \ldots, $x_n$ over $\Q$.
Every algebraic extension of $\Q$ defined by a polynomial $p(x)$ of degree $n$ is obtained as a specialization by substituting the roots of $p(x)$ for $x_1$, \ldots, $x_n$ in the general equation.
The case of the symmetric group was generalized by Emmy Noether in \cite{Noether}, which led to new problems in invariant theory (e.g., Noether problem).

In an analogous way a general linear differential equation with differential Galois group the general linear group $\GL_n(\field)$ over an algebraically closed field $\field$ of characteristic zero is obtained as follows.
Extend the linear action of $\GL_n(\field)$ on the vector whose coordinates are differential indeterminates $y_1$, \ldots, $y_n$ to
the differential field of rational functions $\field \langle y_1, \ldots, y_n \rangle$.
Introducing a new differential indeterminate $Y$ and denoting by $w(y_1, \ldots, y_n)$ the Wronskian of $y_1$, \ldots, $y_n$, the general differential equation is given by (cf., e.g., \cite[Chapter~2]{Magid})
\begin{equation}\label{eq:generaldiffeq}
0 \, = \, Y^{(n)} + c_{n-1} Y^{(n-1)} + \ldots + c_0 Y \, := \,
\frac{w(Y, y_1, \ldots, y_n)}{w(y_1, \ldots, y_n)}.
\end{equation}
In fact, $c_{n-1},\ldots,c_0$ are differentially algebraically independent 
generators of the fixed field of
$\field \langle y_1, \ldots, y_n \rangle$ under $\GL_n(\field)$ (and recall that the elementary symmetric polynomials are algebraically independent generators of the fixed field
of $\Q(x_1, \ldots, x_n)$ under $S_n$).
Every Picard-Vessiot extension of a differential field $\difffield$, with field of constants $\field$, defined by a linear differential polynomial of order $n$ is obtained as a specialization of 
\eqref{eq:generaldiffeq} by substituting the linearly independent solutions $\eta_1$, \ldots, $\eta_n$ for $y_1$, \ldots, $y_n$.
Generalizations to groups other than $\GL_n(\field)$ were obtained in \cite{Gold} and
\cite{JuanMagid2007}. In all these cases the general differential equation involves $n$ differential indeterminates over $\field$ (apart from $Y$).

A different approach resulting in general differential equations involving $d = \dim(\group)$ differential indeterminates was introduced
by Lourdes Juan in \cite{Juan}
for connected linear algebraic groups $\group$.
Fixing a representation of $\group \subset \GL_n$ the general differential equation in this case is the matrix differential equation 
defined by a general element $A(\boldsymbol{a})$ of the Lie algebra $\liealg$ of $\group$ with differential indeterminates
$\boldsymbol{a} = (a_1, \ldots, a_d)$ over $\field$ as coefficients. It is shown that the function field $\juanext = \field\langle \boldsymbol{a} \rangle(\group)$ of the group $\group(\field\langle \boldsymbol{a} \rangle)$ of $\field\langle \boldsymbol{a} \rangle$-rational points
is a Picard-Vessiot extension of $\field\langle \boldsymbol{a} \rangle$ for this equation with differential Galois group $\group(\field)$. The $d$ algebraically independent rational functions generating $\juanext$ over $\field\langle \boldsymbol{a} \rangle$
as a field are differentially algebraically independent over $\field$ and generate $\juanext$ as a differential field over $\field$. Both fields $\juanext$ and $\field\langle \boldsymbol{a} \rangle$ are purely differentially transcendental extensions of $\field$ of degree $d$.
This implies that her construction
is generic in the following sense. For a differential field $\difffield$ with field of constants $\field$ suppose that $A \in \gl_n(\difffield)$ defines a Picard-Vessiot extension of $\difffield$
with differential Galois group a connected subgroup of $\group(\field)$. Then there are elements 
$\boldsymbol{f}=(f_1,\dots,f_d)$ in $\difffield$
such that the matrix $A$ is gauge equivalent to the matrix $A(\boldsymbol{f})$ obtained by specializing the coefficients $\boldsymbol{a}$ in $A(\boldsymbol{a})$ to $\boldsymbol{f}$.

A third approach to constructing general differential equations for the classical groups was presented in \cite{Seiss}. This approach combines the geometric structure of a classical group $\group$ of Lie rank $l$ with Picard-Vessiot theory and involves only
$l$ differential indeterminates. More precisely, the construction starts 
with a differential field $\field\langle \bv \rangle$ generated by $l$ differential indeterminates $\bv=(v_1,\dots,v_l)$ over $\field$. The general extension field in this approach is   
a Liouvillian extension $\generalext$ of $\field\langle \bv \rangle$ with differential Galois group a fixed Borel subgroup $\borel^-(\field)$ of $\group(\field)$.
As in the case of $\GL_n(\field)$ we construct a fundamental matrix $Y$ and define an action of $\group(\field)$ on it which will then induce an action on $\generalext$.
Fixing a Chevalley basis of $\liealg$, the defining matrix of $\generalext$ is chosen such that its conjugate by a representative $\lweyl$ of the longest Weyl group element is the sum of the Cartan subalgebra, parametrized by $\bv$, and the basis elements of the root spaces corresponding to the simple roots. 
Choosing a fundamental matrix $b \in \borel^-(\generalext)$ generating $\generalext$ over $C\langle \bv \rangle$,
we can construct a matrix 
$u$ in the maximal unipotent subgroup of $\borel^-(\field\langle \bv \rangle)$ such that the logarithmic derivative of 
$Y= u \lweyl b$
is the matrix $A_{\group}(\bs)$ constructed in \cite{Seiss} and \cite{Seiss_Generic}. The differential polynomials $\bs=(s_1,\dots,s_l)$ in $\field\{ \bv \}$ are differentially algebraically independent over $\field$. The matrix $u$ is the product of matrices of root groups corresponding to all negative roots $\roots^-$ and it  depends on $|\roots^-|$ differential polynomials $\boldsymbol{p}$ in $\field\{\bv \}$, those corresponding to the negative simple roots being the indeterminates $\bv$. Multiplying $Y$ from the right by an element of the full group $\group(\field)$ and then taking the Bruhat decomposition defines 
an action on $\boldsymbol{p}$ and on the generators of the Liouvillian extension, i.e.\ the entries of $b$, and therefore on $\generalext$. The fixed field under the induced action of $\group(\field)$ on $\generalext$ is $\field\langle \bs  \rangle$ and it is shown that 
the extension $\generalext$ of $\field\langle \bs \rangle$ is a Picard-Vessiot extension with differential Galois group $\group(\field)$.
The construction is only generic for Picard-Vessiot extensions of $\difffield$ with defining matrix gauge equivalent to a matrix in \emph{normal form}, i.e.\ a specialization of $A_{\group}(\bs)$. Deciding such a gauge equivalence is non-trivial as a consequence of the fact that $\generalext$ and $C\langle \bs \rangle $ have differential transcendence degree $l$ over $C$.

The purpose of this paper is to contribute to the question of 
the genericity properties of the general extension  $\generalext$ over $C\langle \bs \rangle$.
We study the relationship of the two general extensions $\generalext/\field\langle \bs \rangle$ and $\juanext / \field\langle \boldsymbol{a} \rangle$ in detail. Clearly the matrix $A_{\group}(\bs)$ can be obtained by a specialization of $A(\boldsymbol{a})$ and so Juan's general extension $\juanext /  C\langle \boldsymbol{a} \rangle$ is 
generic for $\generalext/\field\langle \bs \rangle$. For the opposite direction we need to solve the problem of gauge equivalence of $A(\boldsymbol{a})$ to a matrix in normal form. Over the differential field $\field\langle \boldsymbol{a} \rangle$ this gauge equivalence does not hold in general\footnote{A simple counterexample is the case $\group = \SL_2$, where gauge transformation of $A(\boldsymbol{a})$ into normal form is only possible over an extension $\algext$ of $\field\langle \boldsymbol{a} \rangle$, for example, a quadratic extension.}. Therefore we construct a differential extension field allowing gauge transformation into normal form and preserving the differential Galois group.
Our main theorem is

\begin{theorem}[Theorem~\ref{thm:gauge_eq_&_full_group}] \label{thm:introduction}
There is a differential extension field $\algext$ of 
$C\langle \boldsymbol{a} \rangle$ with field of constants $\field$ 
such that $A(\boldsymbol{a})$ is 
gauge equivalent over $\algext$ to a specialization of $A_{\group}(\bs)$ and the differential 
Galois group of $A(\boldsymbol{a})$ over $\algext$ is the full group $\group(\field)$.
\end{theorem}
\begin{figure}[h]
\centering
\begin{tikzpicture}[thick,scale=0.75]
\node[circle, draw, fill=black, inner sep=0pt, minimum width=4pt] (BB) at (5,1) {};
\node[circle, draw, fill=black, inner sep=0pt, minimum width=4pt] (EE) at (5,3) {};
\node[circle, draw, fill=black, inner sep=0pt, minimum width=4pt] (A) at (0,2) {};
\node[circle, draw, fill=black, inner sep=0pt, minimum width=4pt] (AA) at (1,3) {};
\node[circle, draw, fill=black, inner sep=0pt, minimum width=4pt] (C) at (0,0) {};
\node[circle, draw, fill=black, inner sep=0pt, minimum width=4pt] (CC) at (1,1) {};
\node (D) at (0,1) {};
\node (DD) at (5,2) {};
\node (E) at (1,2) {};
\draw (A) -- (C);
\draw (AA) -- (CC);
\draw (BB) -- (EE);
\draw (A) -- (AA);
\draw (C) -- (CC);
\node (LA) [above left=0.05em of A] {$\juanext$};
\node (LC) [below left=0.05em of C] {$\field\langle \boldsymbol{a} \rangle$};
\node (RAA) [above=0.25em of AA] {$\juanext \cdot \algext$};
\node (RCC) [below=0.25em of CC] {$\algext$};
\node (REE) [above=0.25em of EE] {$\generalext$};
\node (RBB) [below=0.25em of BB] {$\field\langle \bs \rangle$};
\node (LD) [left=0.5em of D] {$\group(\field)$};
\node (RE) [right=0.5em of E] {$\group(\field)$};
\node (RDD) [right=0.5em of DD] {$\group(\field)$};
\draw [decorate,
    decoration = {brace,mirror}] (-0.2,1.8) --  (-0.2,0.2);
\draw [decorate,
    decoration = {brace}] (1.2,2.8) --  (1.2,1.2);  
\draw [decorate,
    decoration = {brace}] (5.2,2.8) --  (5.2,1.2);     
\draw[->, dashed] (4.5,1) to (1.5,1);  
\draw[->, dashed] (4.5,3) to (1.5,3);  
\end{tikzpicture}
\caption{Relation between extensions $\generalext/\field\langle \bs \rangle$ and $\juanext \cdot \algext/\algext$ via specialization}\label{fig:diagram}
\end{figure}
According to the theorem we obtain, after translating the generic extension $\juanext/\field\langle \boldsymbol{a} \rangle$ by $\algext$, the Picard-Vessiot extension $\juanext \cdot \algext/\algext$ as 
a specialization of the general extension $\generalext/\field\langle \bs \rangle$. But this does not imply that the genericity properties of $\juanext/\field\langle \boldsymbol{a} \rangle$ carry over to $\generalext / \field\langle \bs \rangle$.
A reason for this is that, in contrast to $\field\langle \boldsymbol{a} \rangle$, the extension $\algext$ is not purely differentially transcendental over $\field$ and so specializations of the extension $\juanext \cdot \algext / \algext$ need to respect the differential algebraic relations of the generators of $\algext$.
We address how one may overcome this problem in case of a specialization $\extfield / \difffield$ of $\juanext/\field\langle \boldsymbol{a} \rangle$ with differential Galois group $\group(\field)$.
We present in Theorem~\ref{thm:specialization} a sufficient condition on a defining matrix in $\liealg(\difffield)$ ensuring that there exists a differential extension field $\nongenext$ of $\difffield$ with field of constants $\field$ such that
$E \cdot \nongenext / \nongenext$ is a specialization of $\generalext / \field\langle \bs \rangle$ with the same differential Galois group.
In case of $\SL_3$ we show how one can guarantee that this condition is satisfied.

More recently, Man Cheung Tsui has shown in \cite{Tsui} that the differential transcendence degree of a generic Picard-Vessiot extension  for $\group$ with constants $C$ is bounded below by the essential dimension of any Picard-Vessiot  extension with differential Galois group $\group$
and the same field of constants. 
More precisely, let $E/F$ and  $E_0/F_0$ be Picard-Vessiot extensions with differential Galois group $\group(\field)$ such that $E_0\subset E$ and $F_0\subset F$. If $F\cdot E_0 =E$, then $F_0$ is called a differential field of definition for $E/F$. If $A$ is a defining matrix for $E/F$, then the differential field generated over $\field$ by the entries of any gauge equivalent matrix is a differential field of definition of $E/F$.
The essential dimension is now defined as the smallest differential transcendence degree over $\field$ of a field of definition $F_0 \subseteq F$ of the extension. Enlarging $F$ without changing the differential Galois group might have the effect that the essential dimension drops. 
Our concept of genericity features this possibility, leaving it open whether $\generalext/\field\langle \bs \rangle$ is generic in our sense, 
while Man Cheung Tsui's results show that, as expected, $\generalext/\field\langle \bs \rangle $ is not generic in the sense of \cite[Def.~5.1]{Tsui}.

After recalling the structure of the classical groups and fixing notation in Section~\ref{sec:classicalgroups}, basic properties of the gauge transformation are summarized in Section~\ref{sec:gauge_equiv}.
We investigate the possibility of gauge transformation to normal form by using differential elimination in Section~\ref{sec:res_weyl_diff_systems}. For that reason we recall the fundamental notion of Thomas decomposition in Section~\ref{sec:thomas}.
The construction of the field $\algext$ which allows gauge transformation to normal form is carried out in Section~\ref{sec:NF_generic} and the existence of an intermediate differential field $\field\langle \boldsymbol{a} \rangle \subset \algext_2 \subset \algext$ is proved such that $\algext$ is generated over $\algext_2$ by the coefficients $\bt$ of the normal form, which are differentially algebraically independent over $\algext_2$.
In Section~\ref{sec:NF_Galois} we prove that, even after extending the coefficient field from $\field\langle \bv \rangle$ to $\difffield\langle \bv \rangle$ in the construction of \cite{Seiss}, the differential field extension defined by the normal form has the expected Galois group, which completes the proof of Theorem~\ref{thm:introduction}.
In the final section we generalize the earlier discussion of purely differentially transcendental extensions to more general extensions by introducing a condition ensuring that the analogously constructed differential field extension has the correct subfield of constants and the correct Galois group. We also demonstrate how this condition is met in the case $\group = \SL_3$.

\section{The Structure of the Classical Groups}\label{sec:classicalgroups}
Let $\group \subset \GL_n(C)$ be one of the classical groups of Lie rank $l$.  
We denote by
$\roots$ the root system of the corresponding type of $\group$ and 
by $\rootbasis=\{ \alpha_1,\dots,\alpha_l \}$
a basis of $\roots$ with simple roots $\alpha_i$. Each root in $\roots$
can be written as a $\Z$-linear combination of the 
simple roots with all coefficients negative or positive. This 
yields a disjoint decomposition of $\roots$ into a set of positive and 
negative roots, i.e.
\begin{equation*}
\roots \, = \, \roots^+ \cup \roots^-.
\end{equation*}
For a root $\alpha \in \roots$ we write $\height(\alpha)$ for the height of $\alpha$, 
i.e.~the sum of all coefficients in the $\Z$-linear combination with 
respect of $\rootbasis$. 
Further denote by $m$ the cardinalities of $\roots^+$ and $\roots^-$. 
We enumerate the positive roots as $\alpha_1$, \ldots, $\alpha_l$, $\alpha_{l+1}$, \ldots, $\alpha_m$
in such a way that $\alpha_1$, \ldots, $\alpha_l$ are the simple roots introduced above and
$\height(\alpha_r) \leq \height(\alpha_s)$ for all $r \le s$.

We denote by $\weyl$ the Weyl group of $\roots$. It is generated by the $l$ 
simple reflections $w_{\alpha_i}$ for the simple roots $\alpha_i$
of $\rootbasis$. Every $w \in \weyl$ can be written as a product of 
simple reflections and if this expression is minimal we call 
the number of simple reflections the length $l(w)$ of $w$. The length 
coincides with the number of positive roots $\alpha \in \roots^+$ 
with $w(\alpha)$ negative. There is a unique element $\lweyl$ of $\weyl$ 
of maximal length and this element maps all positive roots to negative
ones and vice versa.
We write $\liealg \subset \gl_n(C)$ for the Lie algebra of $\group$ and we let
$\cartan$ be the Cartan subalgebra of $\liealg$ 
consisting of diagonal matrices. Let 
\begin{equation*}
\liealg \, = \, \cartan \oplus
\bigoplus_{\alpha \in \roots} \liealg_{\alpha}
\end{equation*}
be a Cartan decomposition of $\liealg$, where $\liealg_{\alpha}$
denotes the one-dimensional root space corresponding to the root 
$\alpha$ of $\roots$. In accordance with this decomposition we fix a 
Chevalley basis 
\begin{equation}\label{eq:ChevalleyBasis}
\{ H_i, X_{\alpha} \mid 1\leq i \leq l, \ \alpha \in \roots \},
\end{equation}
where $\cartan=\langle H_1 , \dots , H_l \rangle$ and 
$\liealg_{\alpha}= \langle X_{\alpha} \rangle$. 

We denote by $\torus$ the maximal torus of $\group$ with Lie algebra $\cartan$
and for a root $\alpha \in \roots$ we denote by $\unipotent_{\alpha}$ the 
root group whose Lie algebra is the root space $\liealg_{\alpha}$. Let
\begin{equation*}
\exp\colon \liealg_{\alpha}  \rightarrow \unipotent_{\alpha}, 
\ X_{\alpha} \mapsto \sum_{j \geq 0} \frac{1}{j!}  X_{\alpha}^j
\end{equation*}
be the exponential map from the root space $\liealg_{\alpha}$ to the root group $\unipotent_{\alpha}$ of $\group$. For $x \in \field$ we write $u_{\alpha}(x)$ for 
$\exp(x X_{\alpha})$.
Let $\unipotent^{+}$ (respectively $\unipotent^{-}$) be the maximal unipotent subgroup whose Lie
algebra $\mathfrak{u}^+$ (respectively $\mathfrak{u}^-$) 
is the direct sum of all root spaces that are indexed by the positive
(respectively negative) roots. Further we denote 
by $\borel^+$ (respectively $\borel^-$) the Borel subgroup of $\group$ 
with Lie algebra $\mathfrak{b}^+= \cartan \oplus \mathfrak{u}^+$
(respectively $\mathfrak{b}^-=\cartan \oplus \mathfrak{u}^-$).
Clearly we have $\torus < \borel^+$ and $\borel^{+} =\torus \unipotent^{+}$ 
(respectively $\torus < \borel^-$ and $\borel^{-} = \torus \unipotent^{-}$). 
Moreover, 
to shorten notation we will omit the plus sign in 
the notation of the subgroups and algebras,
i.e., we will write $\borel$ for $\borel^+$, etc. 

The following theorems provide a normal form for elements of $\group$ parametrized by $\borel$ and $\weyl$   
(cf.\ \cite[28.3 Theorem and 28.4 Theorem]{HumGroups}).  
\begin{theorem}[Bruhat decomposition] We have $\group= \bigcup_{w \in \weyl} \borel w \borel$ (disjoint union) with
 $\borel w \borel = \borel \widetilde{w} \borel$ if and only if $w = \widetilde{w}$ in $\weyl$.
\end{theorem}
For each $w \in \weyl$ fix a coset representative $n(w)$ in the normalizer of $\torus$
 in $\group$.
\begin{theorem}\label{Bruhat1}
Each element $g \in \group$ can be written in the form 
 \begin{equation*}
 g \, = \, u' \, n(w) \, t \, u\,,
 \end{equation*}
 where $w \in \weyl$, $t \in \torus$, $u  \in \unipotent$ and $u' \in \unipotent'_{w}=\unipotent \cap n(w) \unipotent^{-} n(w)^{-1}$
 are all determined uniquely by $g$. 
\end{theorem}

In the following we will consider the group $\group$ over different fields.
For a field $K$ of characteristic zero, $\group(K)$ is the group of 
$K$-rational points.

\section{Gauge Equivalence}
\label{sec:gauge_equiv}
For $g \in \group \subset \GL_n$ we denote by $\Ad(g)$ the automorphism
\begin{equation*}
 \Ad(g): \liealg \rightarrow \liealg, \ X \mapsto g X g^{-1} , 
\end{equation*}   
i.e., $\Ad$ is the adjoint action of $\group$ on its Lie algebra $\liealg$. Moreover, for a differential field $\difffield$ with derivation $\partial_{\difffield}$ the logarithmic derivative $\dlog$ is defined by
 \begin{equation*}
 \dlog\colon \GL_n(\difffield) \rightarrow \gl_n(\difffield), \ g \mapsto \partial_{\difffield}(g) g^{-1}.
 \end{equation*}
Two matrices $A_1$ and $A_2$ in $\liealg(\difffield)$ are \emph{gauge equivalent} over $\difffield$ by $g \in \group(\difffield)$ if
\begin{equation*}
\gauge{g}{A_1} := \Ad(g)(A_1) + \dlog(g) = A_2.
\end{equation*}
The action of $G(\difffield)$ on $\liealg(\difffield)$ by gauge transformation partitions $\liealg(\difffield)$ into orbits depending on the differential field $\difffield$.
We will use the following two remarks to describe the images of the adjoint action and the logarithmic derivative, respectively. Together they enable us to 
connect gauge transformation of an element of $\liealg$ by a root group element with the root structure. 
\begin{remark}\label{remark3}
  For linearly independent $\alpha$, $\beta \in \roots$ and $r, q \in \N$ let
 $\alpha - r \beta, \dots ,\alpha + q \beta$ be the $\beta$-string through $\alpha$ and let 
 $\langle \alpha, \beta \rangle$ be the Cartan integer. Then, with respect to the Chevalley basis defined in \eqref{eq:ChevalleyBasis}, we have \cite[Section~4.3]{Carter}
  \begin{eqnarray*}
   \mathrm{Ad}(u_{\beta}(x))(X_{\alpha}) & =& \sum\nolimits_{i=0}^q c_{\beta, \alpha,i} x^i X_{\alpha + i \beta}, \\
   \mathrm{Ad}(u_{\beta}(x))(H_{\alpha}) &=& H_{\alpha} - \langle \alpha, \beta \rangle  X_{\beta} , \\
 \mathrm{Ad}(u_{\beta}(x))(X_{-\beta}) &=& X_{-\beta} + x H_{\beta} - x^2  X_{\beta},
\end{eqnarray*}
where $c_{\beta,\alpha,0}=1$ and $c_{\beta,\alpha,i}= \pm \binom{r+i}{i}$.
 \end{remark}

\begin{remark}\label{remark4}
 Let $\group \subset \GL_n$ be a linear algebraic group. Then the restriction of $\dlog$ to $\group$ maps
 $\group(\difffield)$ to its Lie algebra $\liealg(\difffield)$, i.e., we have
 \begin{equation*}
  \dlog|_{\group}: \group(\difffield) \rightarrow \liealg(\difffield).
 \end{equation*}
\end{remark}
 A proof can be found in \cite{KovacicInvProb}.

\section{Thomas Decomposition of Differential Systems}\label{sec:thomas}

In this section we recall the notion of Thomas decomposition of
differential systems to the extent it will be used in later sections.
For more details we refer to, e.g., \cite{RobertzHabil}.

The Bruhat decomposition of classical groups plays a crucial role in our work.
We shall split the task of deciding gauge equivalence to a matrix in normal form into checking consistency of differential systems associated to the Bruhat cells of $\group$ (cf.\ Section~\ref{sec:res_weyl_diff_systems} for the definition of these differential systems $\system_w$ and Definition~\ref{de:nf} for the definition of normal form).
A different application of Thomas decomposition to the Bruhat decomposition of the general linear group was presented in \cite{PleskenThomasBruhat}, where the Bruhat cells are shown to be described by simple algebraic systems.

Let $\difffield$ be a differential field of characteristic zero
and $\diffpolyring = \difffield\{ u_1, \ldots, u_{\ndvar} \}$
the differential polynomial ring over $\difffield$ in the
differential indeterminates $u_1$, \ldots, $u_{\ndvar}$. Differential polynomials
in $u_1$, \ldots, $u_{\ndvar}$ are polynomials in the algebraically independent derivatives
\[
\{ \partial^j u_i \mid i \in \{ 1, \ldots, \ndvar \}, \, j \in \Z_{\ge 0} \},
\]
with coefficients in $\difffield$, where the derivation $\partial$ on $\diffpolyring$ extends the derivation on $\difffield$.
We endow $\diffpolyring$ with a \emph{ranking} $>$, i.e.\ a total ordering $>$ on the above set
such that $\partial^j u_i > u_i$ for all $i$ and all $j \ge 1$, and for all $j_1$, $j_2 \in \Z_{\ge 0}$, the relation
$\partial^{j_1} u_{i_1} > \partial^{j_2} u_{i_2}$ implies $\partial^{j_1 + k} u_{i_1} > \partial^{j_2 + k} u_{i_2}$
for all $k \ge 0$.
A ranking $>$ is said to be \emph{orderly} if $j_1 > j_2$ implies $\partial^{j_1} u_{i_1} > \partial^{j_2} u_{i_2}$ for all $i_1$, $i_2 \in \{ 1, \ldots, n \}$.

\begin{definition}\label{de:differentialsystem}
A \emph{differential system} $\system$, defined over $\diffpolyring = \difffield\{ u_1, \ldots, u_{\ndvar} \}$, is
given by finitely many equations and inequations
\begin{equation}\label{eq:diffsys}
p_1 = 0, \quad p_2 = 0, \quad \ldots, \quad p_s = 0, \quad
q_1 \neq 0, \quad q_2 \neq 0, \quad \ldots, \quad q_t \neq 0,
\end{equation}
where $p_1$, \ldots, $p_s$, $q_1$, \ldots, $q_t \in \diffpolyring$ and
$s$, $t \in \Z_{\ge 0}$.
For a given open subset $\Omega$ of $\C$ the \emph{solution set} of $\system$ is
\begin{eqnarray*}
\Sol_{\Omega}(\system) := \{ \, f = (f_1, \ldots, f_{\ndvar}) & | & f_k\colon \Omega \to \C \mbox{ analytic}, \, k = 1, \ldots, 
\ndvar, \\[0.2em]
& & p_i(f) = 0, \, q_j(f) \neq 0, \, i = 1, \ldots, s, \, j = 1, \ldots, t \, \}.
\end{eqnarray*}
\end{definition}

For each differential polynomial $p \in \diffpolyring \setminus \difffield$ the \emph{leader} of $p$
is the derivative $\partial^j u_i$ that is ranked highest with respect to $>$ among those
occurring in $p$. Differential polynomials are considered recursively as polynomials in their leaders with
coefficients that are differential polynomials in lower ranked derivatives, etc.
The \emph{initial} of $p$ is the leading coefficient of $p$ as a polynomial in its leader.
The \emph{separant} of $p$ is the partial derivative of $p$ with respect to its leader.

By using the above concepts, besides the structure of $\diffpolyring$ as differential ring,
pseudo-reduction of one differential polynomial by another can be defined,
if the leaders of both differential polynomials involve the same differential indeterminate,
which results in a pseudo-remainder (cf., e.g., \cite{RobertzHabil} for details).

\begin{definition}\label{diffsimple}
A differential system $\system$, defined over $\diffpolyring$, as in \eqref{eq:diffsys}
is said to be \emph{simple}
(with respect to $>$) if the following three conditions hold.
\begin{enumerate}
\item \label{diffsimple1}
The system $\system$ is simple as an algebraic system. More precisely,
we consider the finitely many derivatives $\partial^j u_i$ which occur in the equations
and inequations of $\system$ as coordinates of an affine space $\mathbb{A}$ over
an algebraic closure of $\difffield$,
totally ordered by $>$. Then we require that:
\begin{enumerate}
    \item $p_i \not\in \difffield$ for all $i$ and $q_j \not\in \difffield$ for all $j$;
    \item the leaders of $p_1$, \ldots, $p_s$, $q_1$, \ldots, $q_t$ are pairwise distinct;
    \item no initial or separant of $p_1$, \ldots, $p_s$, $q_1$, \ldots, $q_t$ has
    a zero in the algebraic set defined by $p_1 = 0$, \ldots, $p_s = 0$, $q_1 \neq 0$, \ldots, $q_t \neq 0$
    in $\mathbb{A}$.
\end{enumerate}
\item \label{diffsimple2}
No pseudo-reduction of any $p_i$ modulo $p_1$, \ldots, $p_{i-1}$, $p_{i+1}$, \ldots, $p_s$
and their derivatives is possible.
\item \label{diffsimple3}
The left hand sides of the inequations $q_1 \neq 0$, \ldots, $q_t \neq 0$
equal their pseu\-do-re\-main\-ders modulo $p_1$, \ldots, $p_s$
and their derivatives.
\end{enumerate}
\end{definition}

\begin{definition}\label{de:ThomasDecomposition}
Let $\system$ be a differential system, defined over $\diffpolyring$.
A \emph{Thomas decomposition} of $\system$ (or of $\Sol_{\Omega}(S)$)
with respect to $>$
is a collection of finitely many simple differential systems
$\system_1$, \ldots, $\system_r$, defined over $\diffpolyring$, such that
$\Sol_{\Omega}(\system)$ is the disjoint union of the solution sets
$\Sol_{\Omega}(\system_1)$, \ldots, $\Sol_{\Omega}(\system_r)$.
\end{definition}

If the differential system $\system$ is defined over a computable differential ring,
then the construction of a Thomas decomposition of $\system$ is algorithmic.
Then membership to the radical differential ideal generated by a simple
differential system can be decided effectively.

\begin{proposition}[\cite{RobertzHabil}, Prop.~2.2.50]\label{prop:differentialmembership}
Let $\system$ be a simple differential system, defined over $\diffpolyring$, with
equations $p_1 = 0$, \ldots, $p_s = 0$. Moreover, let $E$ be the
differential ideal of $\diffpolyring$ generated by $p_1$, \ldots, $p_s$ and
define the product $q$ of the initials and separants of all $p_1$,
\ldots, $p_s$. Then $E : q^{\infty}$ is a radical differential ideal.
Given $p \in \diffpolyring$, we have $p \in E : q^{\infty}$ if and only if the
pseu\-do-re\-main\-der of $p$ modulo $p_1$, \ldots, $p_s$
and their derivatives is zero.
\end{proposition}

More generally, if the differential system $\system$ is not
necessarily simple, we have the following.

\begin{proposition}[\cite{RobertzHabil}, Prop.~2.2.72]\label{prop:intersection}
Let $\system$ be a (not necessarily simple) differential system
as in \eqref{eq:diffsys}
and $\system_1$, \ldots, $\system_r$ a Thomas decomposition of $\system$ with respect to
any ranking on $\diffpolyring$. Moreover, let $E$ be the differential ideal of $\diffpolyring$
which is generated by $p_1$, \ldots, $p_s$ and define the product $q$ of $q_1$, \ldots, $q_t$.
For $i \in \{ 1, \ldots, r \}$, let $E^{(i)}$ be the differential ideal
of $\diffpolyring$ generated by the equations in $\system_i$ and define the product $q^{(i)}$
of the initials and separants of all these equations. Then we have
\[
\sqrt{E : q^{\infty}} \, = \, \left( E^{(1)} : (q^{(1)})^{\infty} \right) \cap \ldots \cap
\left( E^{(r)} : (q^{(r)})^{\infty} \right).
\]
\end{proposition}

In what follows, differential systems, defined over $\diffpolyring$, that are simple with respect to a block
ranking $>$ on $\diffpolyring$ will be considered.
More precisely, let $\{ u_1, \ldots, u_{\ndvar} \} = B_1 \cup \ldots \cup B_k$
be a partition of the set of differential indeterminates.
For any $i \ge 1$ let $\difffield\{ B_i, \ldots, B_k \}$ be the differential polynomial ring
over $\difffield$ in the differential indeterminates $B_i \cup \ldots \cup B_k$, considered
as a differential subring of $\diffpolyring = \difffield\{ u_1, \ldots, u_{\ndvar} \}$.
We refer to $>$ as an \emph{elimination ranking with blocks $B_1$, \ldots, $B_k$}
if any derivative involving a differential indeterminate in $B_j$ is ranked higher with respect to $>$
than any derivative involving a differential indeterminate in $B_{j+1} \cup \ldots \cup B_k$.
Then we also write $B_1 \gg B_2 \gg \ldots \gg B_k$.

\begin{proposition}[\cite{RobertzHabil}, Prop.~3.1.36]\label{prop:elimsimple}
Let $\system$ be a differential system as in \eqref{eq:diffsys},
defined over $\diffpolyring$, that is simple with respect to an elimination ranking $>$ as defined above.
Moreover, let $E$ be the differential ideal of $\diffpolyring$ generated by $p_1$, \ldots, $p_s$
and $q$ the product of the initials and separants of all $p_1$, \ldots, $p_s$.
For every $i \in \{ 1, \ldots, k \}$,
let $E_i$ be the differential ideal of $\difffield\{ B_i, \ldots, B_k \}$ generated by
\[
P_i := \{ p_1, \ldots, p_s \} \cap \difffield\{ B_i, \ldots, B_k \}
\]
and let $q_i$ be the product
of the initials and separants of all elements of $P_i$. Then, for every
$i \in \{ 1, \ldots, k \}$, we have
\[
(E : q^{\infty}) \cap \difffield\{ B_i, \ldots, B_k \} = E_i : q_i^{\infty}\,.
\]
\end{proposition}

\section{Resolving Weyl Group Elements and their Associated Differential Systems}
\label{sec:res_weyl_diff_systems}

Let $\bap=(a^+_1,\dots, a^+_m)$, $\baz=(a^0_1,\dots,a^0_l)$ and 
$\bam=(a^-_1,\dots,a^-_m)$ be differential indeterminates over $\field$ and, using the notation introduced in Section~\ref{sec:classicalgroups}, define 
a general element of the Lie algebra $\liealg$ by
\begin{equation}\label{eq:defA}
A := \sum_{i=1}^m a^+_i X_{\alpha_i} + \sum_{i=1}^m a^-_i X_{-\alpha_i}
+ \sum_{i=1}^l a^0_i H_i.
\end{equation}

For further differential indeterminates $\bb=(b_1,\dots,b_m)$ over $\field$ let $u_{-\alpha_i}(b_i)$ be the image 
of $b_i X_{-\alpha_i}$ under the exponential map, that is $u_{-\alpha_i}(b_i)$ is a parametrization of the 
root group $U_{-\alpha_i}$.
We fix $w \in \weyl$. Let $i_1,\dots,i_k$ be the indices such that 
$$w(\roots^+) \cap \roots^-=  \{ -\alpha_{i_1},\dots , -\alpha_{i_k} \} , $$
i.e., exchanging the roles played by $\unipotent^{+}$ and $\unipotent^{-}$ in Theorem~\ref{Bruhat1}, $(\unipotent^-)_w'$ is generated by the root groups
$U_{-\alpha_{i_1}},\dots, U_{-\alpha_{i_k}}$. 
We set $\bb_w=(b_{i_1},\dots,b_{i_k})$ and define
\begin{equation}\label{eqn:def_uwb}
u_w(\bb_w)=u_{-\alpha_{i_1}}(b_{i_1}) \cdots u_{-\alpha_{i_k}}(b_{i_k}).    
\end{equation}
Gauge transformation of $A$ by $n(w) u_w(\bb_w)$ yields
\[
\begin{array}{rcl}
\gauge{(n(w) u_w(\bb_w))}{A} & = &
\Ad( n(w) u_w(\bb_w) ) (A)  + \dlog( n(w) u_w(\bb_w))\\[0.5em]
& = & \displaystyle
\sum_{i=1}^m f^+_{w,i} X_{\alpha_i} + \sum_{i=1}^m f^-_{w,i} X_{-\alpha_i}
+ \sum_{i=1}^l f^0_{w,i} H_i\,,
\end{array}
\]
where $f^+_{w,i}$, $f^-_{w,i}$ and $f^0_{w,i}$ are certain differential polynomials in  $\field\{ \bap,\bam,\baz, \bb_w \}$. 
We define the differential system $\system_w$ by 
\begin{equation*}
\system_w : \quad f^+_{w,l+1}=0, \quad \dots, \quad f^+_{w,m} =0. 
\end{equation*}
The system expresses conditions for the vanishing of
the coefficients of the basis elements $X_{\alpha_{l+1}},\dots, X_{\alpha_{m}}$, i.e.,
modulo the differential ideal $\ideal_w$ of $\field\{ \bap,\bam,\baz,\bb_w \}$
generated by $f^+_{w,l+1}$, \ldots, $f^+_{w,m}$ the matrix $A$ will be gauge equivalent to
\[
\sum_{i=1}^l \overline{f}^+_{w,i} X_{\alpha_i} + \sum_{i=1}^m \overline{f}^-_{w,i} X_{-\alpha_i}
+ \sum_{i=1}^l \overline{f}^0_{w,i} H_i,
\]
where $\overline{f}^{\pm}_{w,i}$, $\overline{f}^0_{w,i}$ denote the images of $f^{\pm}_{w,i}$, $f^0_{w,i}$ in the residue class ring.

The differential ideal $\ideal_w$ of $\field\{ \bap,\bam,\baz,\bb_w \}$
is prime.
Indeed, we gauge transform by root group elements belonging to negative roots. Let $u_{-\alpha_{i_j}}(b_{i_j})$ be one of them and let $c$ be the coefficient of some basis element of height $h$ in the linear representation of
\begin{equation}\label{eq:gaugeequivinterm}
    \gauge{u_{-\alpha_{i_{j-1}}}(b_{i_{j-1}})\dots u_{-\alpha_{i_{1}}}(b_{i_{1}})}{A}\,.
\end{equation}
By Remark~\ref{remark3} the effect of the adjoint action on $c$ is that
we potentially add to $c$ a monomial in $b_{i_j}$ and other coefficients of \eqref{eq:gaugeequivinterm} which belong to basis elements for roots of height greater than $h$.
Further, Remark~\ref{remark4} implies that the effect of the logarithmic derivative of $u_{-\alpha_{i_j}}(b_{i_j})$ is that we add the derivative of $b_{i_j}$ to the coefficient of the basis element $X_{-\alpha_{i_j}}$ in the linear representation of \eqref{eq:gaugeequivinterm}. Altogether the effect of gauge transforming \eqref{eq:gaugeequivinterm} with $u_{-\alpha_{i_j}}(b_{i_j})$ is that we add to the coefficient $c$ differential monomials in the indeterminate $b_{i_j}$ and in those 
indeterminates $\bap,\bam,\baz$ which are of height greater than $h$.
We conclude inductively that for each basis element $X_{\alpha_i}$, $X_{-\alpha_i}$ or $H_i$ of $\liealg$, the differential indeterminate which is the corresponding coefficient in \eqref{eq:defA} appears in the coefficient of the same basis element in
\begin{equation}\label{eq:gaugeequivinterm2}
    \gauge{u_w(\bb_w) }{A}
\end{equation}
only linearly with constant coefficient.
Finally gauge transforming \eqref{eq:gaugeequivinterm2} with $n(w)$ is only a permutation of the entries and so $f^+_{w,l+1}$, \ldots, $f^+_{w,m}$ can be solved 
for pairwise distinct differential indeterminates.
\begin{lemma}\label{lem:prime}
The differential ideal $\ideal_w$ is prime.
\end{lemma}

We compute a Thomas decomposition $\system_w^1$, \ldots, $\system_w^r$ of $\system_w$
with respect to an elimination ranking satisfying
\[
\{ \bb_w \} \gg \{ \bap, \, \bam, \, \baz\,\} \,.
\]
For $i \in \{ 1, \ldots, r \}$ let $\ideal_w^i$ be the differential ideal of 
$\field\{  \bap,\bam,\baz,\bb_w \}$
generated by the left hand sides of the equations in $\system_w^i$, and define $q_w^i$ as the product of the initials and separants of all these generators. Then, by Proposition~\ref{prop:intersection}, we have
\[
\ideal_w = \left( \ideal_w^1 : (q_w^1)^{\infty} \right) \cap \ldots \cap \left( \ideal_w^r : (q_w^r)^{\infty} \right).
\]
Since $\ideal_w$ is prime, there exists a unique $j \in \{ 1, \ldots, r \}$ such
that $\ideal_w = \ideal_w^j : (q_w^j)^{\infty}$. Hence, the simple system $\system_w^j$ is generic
in the sense that the solution set of $\system_w^j$ is a
dense subset of the solution set of $\system_w$ (cf.\ also \cite[Subsect.~2.2.3]{RobertzHabil}).

Let $\mathcal{T}_w$ be the subset of $\system_w^j$ consisting of the equations and inequations
which only involve $\bap$, $\bam$, $\baz$ and their derivatives.
Let $E_w$ be the differential ideal of $\field\{ \bap, \bam, \baz \}$ which is generated by the left hand sides of the equations $\mathcal{T}_w$,
and define $q$ as the product of the initials and separants of these generators.
Because of the choice of elimination ranking we have
\[
\left( \ideal_w^j : (q_w^j)^{\infty} \right) \cap \field\{ \bap, \bam, \baz \} \, = \,
E_w : q^{\infty}.
\]
In principle, this approach allows to decide effectively whether gauge transformation by an element of a Bruhat cell defined by $w$ into normal form is possible by checking whether $E_w : q^{\infty}$ is the zero ideal. Moreover, by trying the finitely many elements of $\weyl$, a feasible one for gauge transformation to normal form can be found if one exists. In fact, we are going to predict such a feasible element and hence prove that gauge transformation to normal form is possible.

\begin{definition}
For $w \in \weyl$ we consider the two conditions:
\begin{equation}\tag{A}\label{eq:conditionA}
\exists \, \Psi \subseteq \roots^{-} \mbox{ such that } w(\Psi) = \roots^{+} \setminus \rootbasis
\end{equation}
and
\begin{equation}\tag{B}\label{eq:conditionB}
w(\roots^{+}) \cap \roots^{-} \supseteq \Psi\,.
\end{equation}
An element $w \in \weyl$ is called \emph{resolving} if it satisfies
Conditions~\eqref{eq:conditionA} and \eqref{eq:conditionB}.
\end{definition}
If $\Psi$ as above exists, then it is uniquely determined by $w$, namely as
$$\Psi = w^{-1}(\roots^+ \setminus \rootbasis).$$
In particular, we have $|\Psi| = m - l$.

\begin{proposition}\label{thm:existresolving}
For every classical group $\group$ the associated Weyl group $\weyl$ contains a resolving element.
\end{proposition}
\begin{proof}
Since the root system of a classical group is irreducible, there is a unique Weyl group element $\lweyl \in \weyl$ which maps $\roots^+$ to $\roots^-$. As a consequence we have $\lweyl (\rootbasis)=\rootbasis^-$ and
because $\lweyl$ has order two, it follows that $\lweyl(\roots^{-})=\roots^{+}$ 
and $\lweyl(\rootbasis^-)=\rootbasis$. Define $\Psi:=\roots^{-} \setminus \rootbasis^{-}$.
One easily checks that $\Psi$ and $\lweyl$ satisfy Condition~\eqref{eq:conditionA} and Condition~\eqref{eq:conditionB}.
\end{proof}

\begin{definition}\label{de:rankingadapted}
A ranking $>$ on $\field\{ \bap, \bam, \baz, \bb \}$
is said to be \emph{adapted to the root system} if the following three conditions are satisfied.
\begin{enumerate}
    \item The ranking $>$ is an elimination ranking satisfying $\{\bb \} \gg \{ \bap, \, \bam, \, \baz\}$.
    \item The restriction of $>$ to $\field\{ \bb \}$ is an orderly ranking.
    \item For all $i$, $j \in \{ 1, \ldots, m \}$ we have 
\[
b_i > b_j \quad \Rightarrow \quad | \height(-\alpha_i) | \ge | \height(-\alpha_j)| \,.
\]
\end{enumerate}
\end{definition}
Obviously a ranking that is adapted to the root system exists.

\begin{theorem}\label{thm:Sw_simple}
Let $w \in \weyl$ be resolving and assume that the ranking $>$ is adapted 
to the root system. Then we have:
\begin{enumerate}
\item \label{stat1:thm:Sw_simple} The differential system 
\[
\system_w : \quad f^+_{w,l+1} = 0, \quad \dots, \quad f^+_{w,m} = 0
\]
    is simple with respect to $>$.
    \item \label{stat2:thm:Sw_simple}
    For $r \in \{l+1,\dots,m \}$ the leader of $f^+_{w,r}$ is $b_{i}'$ 
    corresponding to $-\alpha_{i}=w^{-1}( \alpha_r) \in \Psi$.
    \item \label{stat3:thm:Sw_simple} The differential polynomials $f^+_{w,l+1}$, \dots,$ f^+_{w,m}$ are  linear in their leaders and have constant initials.
    \item \label{stat4:thm:Sw_simple} The differential ideal generated by $f^+_{w,l+1}$, \dots,$ f^+_{w,m}$ is prime.
    \item \label{stat5:thm:Sw_simple} If the leader of $f^+_{w,r}$ is $b_i'$, then $f^+_{w,r}$ contains the variable $a^-_i$ 
linearly with constant coefficient and variables $a^-_j$ with $|\height(-\alpha_j)| <|\height(-\alpha_i)|$
do not appear in $f^+_{w,r}$.
\end{enumerate}
\end{theorem}

\begin{proof}
To begin with we compute the logarithmic derivative of
\[
u(\bb) \, = \, u_{-\alpha_1}(b_1)\cdots u_{-\alpha_m}(b_m).
\]
We shall argue that we obtain
\[
\dlog(u(\bb)) = \sum_{i=1}^m \Ad(u_{-\alpha_1}(b_1) \ldots u_{-\alpha_{i-1}}(b_{i-1})) \, \big(\dlog(u_{-\alpha_i}(b_i)) \big)  = \sum_{i=1}^m (b_i' +  p_i ) X_{-\alpha_i}\,,
\]
where $p_i$  is a differential polynomial in $\bb$ with $ \mathrm{ld}(b_i'+p_i)=b_i'$. By Remark~\ref{remark4} we have 
$\dlog(u_{-\alpha_i}(b_i))=b_i'X_{-\alpha_i}$. By Remark~\ref{remark3} additional terms in the coefficient of $X_{-\alpha_i}$ can only stem from the
adjoint action  
\[
\begin{array}{l}
\Ad(u_{-\alpha_1}(b_1) \ldots u_{-\alpha_{j-1}}(b_{j-1})) \, \big(\dlog(u_{-\alpha_j}(b_j)) \big)\\[0.5em]
\qquad\qquad
= \, \Ad(u_{-\alpha_1}(b_1) \ldots u_{-\alpha_{j-1}}(b_{j-1})) \, \big(b_j' X_{-\alpha_j} \big)
\end{array}
\]
such that $|\height(-\alpha_i)|>|\height(-\alpha_j)|$.
We conclude that
every term in each $p_i$ is of degree at least two and is divisible by precisely one
derivative $b_j'$ for some $j$ with $|\height(-\alpha_i)| > |\height(-\alpha_j)|$. 
We obtain now $u_w(\bb_w)$ from $u(\bb)$ by setting precisely those $b_k$ equal to zero for
which $-\alpha_k \not\in w(\roots^{+}) \cap \roots^{-}=\Psi$ and keeping the other $b_i$ as
indeterminates. Hence, 
for each $-\alpha_i \in \Psi$, the derivative $b_i'$ appears linearly and with constant coefficient in
the coefficient of $X_{-\alpha_i}$. Further, since the ranking is adapted to the root system,
the variable $b_i'$ is the leader of that coefficient. 

Adding the contribution $\Ad(u_w(\bb_w)) (A)$ of the adjoint action to the logarithmic derivative yields the
gauge transformation of $A$ by $u_w(\bb_w)$. Since the adjoint action only creates polynomial entries, we have
\[
\begin{array}{rcl}
\Ad (u_w(\bb_w)) (A) + \dlog(u_w(\bb_w)) & = &
\displaystyle\sum_{-\alpha_i \in \Psi} (b_i' +  p_i +q_i ) X_{-\alpha_i} \, +\\[0.5em]
\multicolumn{3}{c}{\displaystyle
\sum_{-\alpha_i \in \roots \setminus \Psi} g_{-\alpha_i} X_{-\alpha_i} + 
\sum_{\alpha_i \in \roots^{+} } g_{\alpha_i} X_{\alpha_i}
+\sum_{i=1}^l g_{0,i} H_i\,,}
\end{array}
\]
where $q_i$ as well as  $g_{\alpha_i}$ and $g_{0,i}$ are polynomials in the indeterminates 
$\bap$, $\bam$, $\baz$, $\bb_w$ and where $g_{-\alpha_i}$ are differential polynomials in the same indeterminates. 
By the first formula in Remark~\ref{remark3} the polynomial $q_i$ contains the term $a^{-}_i$ linearly and with constant coefficient and all other terms are non-constants and lie in the
ring 
\begin{equation*}
\field[\bap,\baz,\bb_w, a^-_j \mid |\height(-\alpha_i)| > |\height(-\alpha_j)| ].
\end{equation*}
Again, because the ranking is adapted to the root system, the leader in the coefficient 
of $X_{-\alpha_i}$ with $-\alpha_i \in \Psi$ is still $b_i'$.
By Condition~\eqref{eq:conditionA}, after a further gauge transformation by $n(w)$,
these coefficients become the left hand sides in
$$\system_w: \quad f^+_{w,l+1} = 0, \quad \dots, \quad f^+_{w,m} = 0 .$$
It follows that they satisfy statements~\ref{stat2:thm:Sw_simple}, \ref{stat3:thm:Sw_simple} and \ref{stat5:thm:Sw_simple}.   
Since $\system_w$ consists of equations with pairwise distinct leaders
and the left hand sides are linear in their leaders with constant initials,
Condition~\ref{diffsimple1} of Definition~\ref{diffsimple} is satisfied. Moreover, since the
leaders of these equations involve pairwise distinct differential indeterminates,
Condition~\ref{diffsimple2} of Definition~\ref{diffsimple} holds as well.
Since $\system_w$ contains no inequations, we conclude that
the differential system $\system_w$ is simple with respect to $>$.
Finally, statement~\ref{stat4:thm:Sw_simple} of the theorem follows from Lemma~\ref{lem:prime}.
\end{proof}

\begin{remark}
There may be Weyl group elements $w \in \weyl$ which are not resolving, but which lead to a consistent system 
$$\system_w: \quad f^+_{w,l+1} = 0, \quad \dots, \quad f^+_{w,m} = 0.$$
For example, when $\group = \mathrm{SP}_8$, i.e.\ the group of type $C_4$, in the representation defined in \cite[VIII, \S13.3]{BourGroupesetAlgebresVII-VIII},
the Weyl group element 
\[
{\small
w = \left( \begin{array}{cccccccc}
0&0&0&0&0&1&0&0\\
0&0&0&0&0&0&0&-1\\
0&0&0&0&0&0&-1&0\\
0&0&0&-1&0&0&0&0\\
0&0&0&0&-1&0&0&0\\
0&1&0&0&0&0&0&0\\
1&0&0&0&0&0&0&0\\
0&0&-1&0&0&0&0&0
\end{array} \right)
}
\]
is not resolving, but one can check by computing a Thomas decomposition that the corresponding system $\system_w$ is consistent, i.e.\ admits solutions.
\end{remark}

\section{A Normal Form for a Generic Equation}\label{sec:NF_generic}
For a classical group $\group$ with root system $\roots$ and $l$ differential indeterminates $\boldsymbol{t}=(t_1,\dots,t_l)$ over $\field$ we constructed in \cite{Seiss} and \cite{Seiss_Generic} the matrix  
$A_{\group}(\boldsymbol{t})$ which defines a Picard-Vessiot extension of $\field\langle \boldsymbol{t} \rangle$ with differential Galois group $\group(\field)$. For its definition let $A_0^+$ and $A_0^-$ be the sum of basis elements of the root spaces corresponding to the simple roots and the negatives of the simple roots, respectively. It is shown in \cite[Lemma 6.4]{Seiss} that there are $\gamma_1,\dots,\gamma_l \in \roots^+$ such that the sum
$$\mathfrak{b}^+ \, = \, \ad (A_0^-)(\mathfrak{u}^+) + \sum_{i=1}^l \liealg_{\gamma_i} $$
is direct and $\height(\gamma_i)=m_i$, where $m_i$ are the exponents of $\liealg$. We refer to $\gamma_1,\dots,\gamma_l$ as \emph{complementary roots}.
\begin{definition}\label{de:nf}
The matrix 
$$ A_{\group}(\boldsymbol{t}) \, = \, A_0^+ + \sum_{i=1}^l t_i X_{-\gamma_i} $$
is called the \emph{normal form matrix} for $\group$ (with respect to the complementary roots $\gamma_1,\dots,\gamma_l$).
\end{definition} 

Now let $A$ be as in Section~\ref{sec:res_weyl_diff_systems}. From \cite{Juan} it is known that
$A$ defines a Picard-Vessiot extension of $\field\langle \bap, \bam,\baz \rangle$
with differential Galois group $\group(\field)$. 
In this section we construct 
a differential field extension 
$$\field\langle \bap, \bam,\baz \rangle \hookrightarrow \algext$$ such that the following conditions are satisfied:
\begin{enumerate}
\item The field of constants of $\algext$ is $\field$.
\item For the image $A \in \liealg(\algext)$ and for some $\boldsymbol{\overline{t}} \in \algext^l$, the matrix $A$
is gauge equivalent to $A_{\group}(\boldsymbol{\overline{t}})$ over $\algext$.
\item The Galois group of the equation defined by $A_{\group}(\boldsymbol{\overline{t}})$ over $\algext$ is $\group(\field)$.
\end{enumerate}

Let $w$ be a resolving element of the Weyl group $\weyl$ of $\group$ and let 
$>$ be a ranking on $\field\{ \bap, \bam, \baz,\bb_w \}$
which is adapted to the root system.
Then, by Theorem~\ref{thm:Sw_simple},
the differential ideal $\ideal_w$ of $\field\{ \bap, \bam, \baz,\bb_w \}$
which is generated by $f^+_{w,l+1}$, \ldots, $f^+_{w,m}$ is prime.
Furthermore, the generators form a
simple differential system and their leaders are the $b_i'$ corresponding to the
elements $-\alpha_i$ of $\Psi$. Moreover, the differential indeterminate $a^-_i$ occurs linearly with
constant coefficient in the generator with leader $b_i'$ and is the one
with largest $|\height(-\alpha_j)|$ among the differential indeterminates $a^-_j$ occurring in this generator.
Since $a^-_1$, \ldots, $a^-_m$ are differentially algebraically independent over $\field$, we have
\begin{equation}\label{eq:no_si}
\ideal_w \cap \field\{ \bap, a^-_i, \baz, \bb_w \mid -\alpha_i \not\in \Psi, i = 1, \ldots, m \} \, = \, \{ 0 \}.
\end{equation}
The kernel of the natural homomorphism
\[
\nu\colon \field\{ \bap, \bam, \baz, \bb_w \} \longrightarrow
\resfield := \Frac(\field\{ \bap, \bam, \baz, \bb_w  \} / \ideal_w)
\]
is equal to $\ideal_w$. We denote the images of the indeterminates by 
\[
\bi_i := \nu(b_i), \quad
\azi_i := \nu(a^0_i), \quad
\api_i := \nu(a^+_i), \quad
\ami_i := \nu(a^-_i) 
\]
and we write $\bazi=(\azi_1, \dots, \azi_{l})$, 
$\bapi=(\api_1, \dots , \api_m)$, $\bami=(\ami_1, \dots, \ami_m)$.
Moreover, denote by $\baml$ and $\bamil$ the tuples of the $l$ elements $a^-_i$
and $\ami_i$, respectively, with the property that $-\alpha_i \notin \Psi$.
Since
\begin{equation}\label{eq:no_bi}
\ideal_w \cap \field\{ \bam, \bap, \baz \} \, = \, \{ 0 \},
\end{equation}
the elements $\bazi$, $\bapi$ and $\bami$ are differentially algebraically
independent over $\field$. We can identify the differential field 
$\field\langle \bam,\bap,\baz \rangle$ with the differential subfield 
$\field\langle  \bami, \bapi,\bazi \rangle$ of $\resfield$, i.e.,  
$\resfield$ is a differential extension field of $\field\langle   \bami, \bapi,\bazi \rangle$ which is a purely differentially transcendental differential field 
 over $\field$.
We consider the differential subfield $\field\langle \bbi_w \rangle$ of $\resfield$, where $\bbi_w =(\bi_{i_1},\dots,\bi_{i_{m-l}})$ and  
$\Psi = \{-\alpha_{i_1},\dots, - \alpha_{i_{m-l}} \} $.

\begin{lemma}\label{DiffIndep1}
The subset $\{ \bazi, \bapi, \bamil \}$
of $\resfield$ is differentially algebraically independent over $\field\langle \bbi_w \rangle$
and $$\field\langle \bbi_w \rangle \langle \bazi, \bapi,\bamil \rangle \, = \, \resfield.$$
\end{lemma}
\begin{proof}
Let $P(\mathbf{Z}) \in \field\langle \bbi_w \rangle\{ \mathbf{Z} \}$
be a differential polynomial in the differential indeterminates $\mathbf{Z}=(Z_1,\dots,Z_{m+2l})$ such that $P( \bazi, \bapi, \bamil) = 0$.
After multiplying $P(\mathbf{Z})$ by a non-zero element of $\field\langle \bbi_w \rangle$ we may assume
that $P(\mathbf{Z})$ lies in the image of the composition of homomorphisms
\[
\field\{ \bb_w \}\{\mathbf{Z} \} \longrightarrow
\field\{ \bb_w, \baz, \bap, \bam \}\{ \mathbf{Z} \} \longrightarrow
\resfield\{ \mathbf{Z} \},
\]
where the first homomorphism is induced by the inclusion 
$$\field\{ \bb_w \} \longrightarrow \field\{ \bb_w, \baz, \bap, \bam \}$$
and the second homomorphism is defined by applying $\nu$ to each coefficient of a differential polynomial.
Let $\widehat{P}( \mathbf{Z}) \in \field\{ \bb_w \}\{  \mathbf{Z} \}$ be any preimage of
$P( \mathbf{Z})$ under the above composition of homomorphisms.
Note that by definition of $\bbi_w$ and $\bazi$, $\bapi$, $\bami$
the diagram
\[
\begin{tikzcd}
\field\{ \bb_w \}\{ \mathbf{Z} \} \arrow{rr}{b_i \, \mapsto \, \bi_i}
\arrow[swap]{ddr}{\widehat{P}(\mathbf{Z}) \, \mapsto \,  \nu(\widehat{P}(\baz, \bap, \baml))} & & \field\langle \bbi_w \rangle\{ \mathbf{Z} \} \arrow{ddl}{P(\mathbf{Z}) \, \mapsto \, P(\bazi, \bapi, \bamil)} \\ & & \\
  & \resfield &
\end{tikzcd} 
\]
is commutative.
Hence, $P(\bazi, \bapi, \bamil) = 0$
implies $\widehat{P}(\baz, \bap, \baml) \in \ker(\nu) = \ideal_w$.
Now, since 
$$\widehat{P}(\baz, \bap, \baml) \in \ideal_w \cap \field\{ \bb_w, \baz, \bap, \baml \},$$ 
we conclude with \eqref{eq:no_si} that $\widehat{P}(\mathbf{Z})$
is the zero polynomial and so is $P(\mathbf{Z})$.

The differential field $\resfield$ is generated by $\bbi_w$ and $\bazi$, $\bapi$ 
and $\bami$ over $\field$. We prove now that $\ami_i$ lies in
$\field\langle \bbi_w \rangle \langle \bazi, \bapi, 
\bamil \rangle$
for every $-\alpha_i \in \Psi$, i.e., $\resfield$ is generated by the above elements
with $\bamil$ instead of $\bami$. We fix some $-\alpha_i \in \Psi$. 
Using $f^+_{w,l+1},\dots,f^+_{w,m}$ and their properties stated in Theorem~\ref{thm:Sw_simple}
we can express $\ami_i$ in terms of 
$\bbi_w$, $\bazi$, $\bapi$, $\bamil$ 
and $\ami_j$ with $-\alpha_j \in \Psi$ of height greater than the height of $-\alpha_i$, proceeding recursively and starting with the 
elements of greatest height.
\end{proof}

\begin{lemma}\label{lem:fieldofconstants}
The field of constants of $\resfield$ is $\field$.
\end{lemma}

\begin{proof}
By Lemma~\ref{DiffIndep1} we have $\resfield=\field\langle \bbi_w, \bazi, \bapi, \bamil \rangle $. Then  
the field of constants of $\resfield$
is $\field$, because $\bbi_w$, $\bazi$, $\bapi$, $\bamil$
are differentially algebraically independent over $\field$. Indeed, if $P(\mathbf{Z})$ is a differential polynomial
over $\field$ in the indeterminates $\mathbf{Z}=(Z_1, \ldots, Z_{2m+l})$ such that $P(\bbi_w, \bazi, \bapi, \bamil) = 0$,
then $P(\bb_w, \baz, \bap, \baml)$ can be considered as an element of $\field\{ \bb_w, \baz, \bap, \bam \}$
which is in the kernel of the natural homomorphism $\nu$. But \eqref{eq:no_si} implies that $P(\mathbf{Z})$ is the zero polynomial.
\end{proof}

The gauge transformation of $A(\bapi,\bami,\bazi)$ 
to a matrix $A_G(\bt)$ in normal form
with $\bt=(\overline{t},\dots,\overline{t}_l)$ in some finite algebraic extension of $\resfield$ is done in three steps. First we show that $A(\bapi,\bami,\bazi)$ is gauge
equivalent to a matrix $A(\bhp,\bhm,\bhz)$ with 
$\bhp=(\overline{h}_1^+,\dots,\overline{h}^+_l,0,\dots,0)$ over $\resfield$ and $\overline{h}^+_1,\dots,\overline{h}^+_l$, 
$\bhm$, $\bhz$ are differentially algebraically independent over 
$\field\langle \bbi \rangle$ and 
$\field\langle \bbi \rangle \langle \overline{h}^+_1,\dots, \overline{h}^+_l,\bhm,
\bhz \rangle = \resfield$.
In a second step we prove that $A(\bhp,\bhm,\bhz)$ is 
gauge equivalent to a matrix  
$A(\bgp,\bgm,\bgz)$ with 
$\bgp=( 1,\dots,1,0,\dots,0)$
with $l$ entries equal to $1$ and $m-l$ entries equal to $0$,
over an algebraic extension $\algext$ of $\resfield$. It will turn out that there is a differential subfield $\algext_1$ 
 of $\algext$ such that $\bgm,\bgz$ are differentially 
 algebraically independent over $\algext_1$ and 
 $\algext_1\langle \bgm,\bgz \rangle = \algext$.  
In a last step we gauge transform $A(\bgp,\bgm,\bgz)$ to 
a matrix of shape $A_G(\bt)$ and we prove that there is a differential subfield $\algext_2$ such that 
$\algext_2 \langle \bt \rangle = \algext$ and $\bt$ is 
differentially algebraically independent over $\algext_2$. Applying Theorem~\ref{thm:difftransfullgroup} 
we obtain that the differential Galois group of $A_G(\bt)$ over $\algext$ is
the full group $\group$ and so is the one for $A(\bapi,\bami,\bazi)$ 
over $\algext$.

Extending the notation of \eqref{eqn:def_uwb} let
\begin{equation*}
\boldsymbol{u}(\bbi_w) \, = \, u_{-\alpha_{i_1}}(\overline{b}_{i_1}) \cdots  u_{-\alpha_{i_{m-l}}}(\overline{b}_{i_{m-l}}) 
\end{equation*}
with the previously fixed
resolving Weyl group element $w$.

\begin{lemma}\label{TransStep1}
Gauge transformation yields
\[
\gauge{(n(w)\boldsymbol{u}(\bbi_w))}{A(\bapi,\bami,\bazi)} \, = \, A(\bhp,\bhm,\bhz)
\]
with $\bhp=(\overline{h}^+_1,\dots,\overline{h}^+_l,0,\dots,0) \in \resfield^m$,
$\bhm \in \resfield^m$ and 
$\bhz \in \resfield^l$. Further we have 
$$\field\langle \bbi_w \rangle \langle \bhp,\bhm,
\bhz \rangle \, = \, \resfield $$ 
and $\overline{h}^+_1,\dots,\overline{h}^+_l$, $\bhm$,
$\bhz$ are differentially algebraically independent over 
$\field\langle \bbi_w \rangle$.
\end{lemma}
\begin{proof}
First we apply gauge transformation by $\boldsymbol{u}(\bbi_w)$ to $A(\bapi,\bami,\bazi)$ 
and obtain 
\[
\gauge{\boldsymbol{u}(\bbi_w)}{A(\bapi,\bami,\bazi)} \, = \,
\sum_{i=1}^m \overline{p}^+_{i} X_{\alpha_i} + \sum_{i=1}^m \overline{p}^-_{i} X_{-\alpha_i}  +  
\sum_{j=1}^l \overline{p}^0_{j} H_j\,,
\]
where $\overline{p}^{\pm}_{i}$, $\overline{p}^0_{j} \in \field\{\bbi_w\} \{\bapi,\bazi,\bami \} \subset \resfield$.
First of all, by construction of the extension field $\resfield$ and its elements $\bbi_w$, the coefficients $\overline{p}^-_{i}$ with $-\alpha_i \in \Psi$ vanish. 
Let $\mathbf{\overline{p}}^+=(\overline{p}^+_1,\dots,\overline{p}^+_m)$ and  $\mathbf{\overline{p}}^0=(\overline{p}^0_1,\dots,\overline{p}^0_m)$ and write $\mathbf{\overline{p}}^-_{\neg \Psi}$ for the $l$-tuple 
of those $\overline{p}^-_i$ such that $-\alpha_i \notin \Psi$. 
By Lemma~\ref{DiffIndep1} we have $\resfield = \field\langle \bbi_w \rangle \langle \bazi, \bapi, 
\bamil \rangle $ and we are going to prove that 
 $$ \field\langle \bbi_w \rangle \langle \bazi, \bapi, 
\bamil \rangle \, = \, \field\langle \bbi_w \rangle \langle \mathbf{\overline{p}}^+ , \mathbf{\overline{p}}^0, 
\mathbf{\overline{p}}_{\neg \Psi}^- \rangle $$
 by showing that $\bazi, \bapi, 
\bamil$ are contained in $\field\langle \bbi_w \rangle \langle \mathbf{\overline{p}}^+ , \mathbf{\overline{p}}^0, 
\mathbf{\overline{p}}_{\neg \Psi}^- \rangle$. Then we conclude that $\mathbf{\overline{p}}^+ , \mathbf{\overline{p}}^0, \mathbf{\overline{p}}_{\neg \Psi}^-$ are 
differentially algebraically independent over $\field\langle \bbi_w \rangle$, since 
by Lemma~\ref{DiffIndep1}  
the same is true for $\bapi,\bazi,\bamil$ (by the invariance of the differential transcendence degree).

The desired equality of the differential fields follows from the fact that one can recursively
solve the differential polynomials $\mathbf{\overline{p}}^+ , \mathbf{\overline{p}}^0, \mathbf{\overline{p}}_{\neg \Psi}^-$
for the indeterminates  $\bapi$, $\bazi_i$, $\bamil$, starting 
with the indeterminate corresponding to the root of greatest height.
More precisely, since we gauge 
transformed $A(\bapi,\bami,\bazi)$ with $\boldsymbol{u}(\bbi_w)$, i.e. a product of
root group elements corresponding to negative roots only, the differential polynomial 
$p^+_{m}$, which belongs to the greatest positive root $\alpha_m$, is simply the indeterminate $\api_m$. Hence, we trivially have  
$\api_m \in \field\langle \bbi_w \rangle \langle  \mathbf{\overline{p}}^+ , 
\mathbf{\overline{p}}^0, \mathbf{\overline{p}}_{\neg \Psi}^- \rangle$.
Let $\overline{p}_i$ be one of the coefficients corresponding to a root of height less than $m$ and let $\overline{a}_i$ be the  indeterminate among  $\bazi, \bapi, \bamil$ corresponding to the same root.
Then in $\overline{p}_i$ the indeterminate $\overline{a}_i$
 occurs linearly with non-zero constant coefficient.
 Moreover, the indeterminates other than $\overline{a}_i$ among $\bapi$, $\bazi$, $\bamil$ occurring in $\overline{p}_i$
 correspond to roots of height neither smaller than nor equal to the one corresponding to $\overline{a}_i$, where we consider the indeterminates $\bazi$ to correspond to height zero. So we recursively obtain that $\bapi$, $\bazi$, 
$\bamil$ lie in $\field\langle \bbi_w \rangle \langle \mathbf{\overline{p}}^+ , \mathbf{\overline{p}}^0, 
\mathbf{\overline{p}}_{\neg \Psi}^- \rangle$ as claimed.

Finally the gauge transformation of the intermediate matrix with $n(w)$ yields a matrix whose
coefficients $\bhp$, $\bhm$ and $\bhz$ have the desired properties.
\end{proof}

\begin{lemma}\label{TransStep2}
There exists a finite algebraic extension $\algext$ of $\resfield$ and an element $t$ of the torus $\torus(\algext)$ of $\algext$-rational points such that  
$A(\bhp,\bhm,\bhz)$ is gauge 
equivalent by $t$ to a matrix of shape 
$A(\bgp,\bgm,\bgz)$ with 
\begin{equation*}
\bgp \, = \, (\underbrace{1,\dots,1}_{l}, \underbrace{0,\dots,0}_{m-l}), 
\end{equation*}
$\bgm \in \algext^m$ and $\bgz \in \algext^l$. 
There is a differential subfield $\algext_1$ of $\algext$ such that  $\algext_1\langle \bgm,
\bgz \rangle = \algext$ and $\bgm$ and $\bgz$ are differentially 
 algebraically independent over $\algext_1$.
\end{lemma}
\begin{proof}
For indeterminates $\boldsymbol{x}=(x_1,\dots, x_l)$ let 
\[
\boldsymbol{t}(\boldsymbol{x}) \, = \, t_1(x_1) \cdots t_l(x_l)\,,
\]
where for a simple root $\alpha_i$ the torus element $t_i(x_i)$ is the image of $\mathrm{diag}(x_i,x_i^{-1})$ under the homomorphism from $\SL_2$ to the group
generated by $U_{\alpha_i}$ and $U_{-\alpha_i}$. For two simple roots $\alpha_i$, $\alpha_j \in \rootbasis$
we have
$$\Ad (t_i(x_i))(X_{\alpha_j}) \, = \, x_i^{\langle \alpha_i,\alpha_j\rangle } X_{\alpha_j}\,,$$
where $\langle \alpha_i,\alpha_j\rangle$ is the integer at the entry $(i,j)$ of the corresponding Cartan matrix
and so the coefficient of $X_{\alpha_j}$ in the linear representation of  
\[
\Ad(\boldsymbol{t}(\boldsymbol{x}))(A(\bhp,\bhm,\bhz))
\]
is $\overline{h}^+_j  x_1^{\langle \alpha_1,\alpha_j\rangle} \cdots x_l^{\langle \alpha_l,\alpha_j\rangle}$. 
We require now that 
\begin{equation}\label{eq:normalization}
\overline{h}^+_1 \prod_{i=1}^l x_i^{\langle \alpha_i,\alpha_1\rangle} \, = \, 1 ,\quad \dots, \quad \, \overline{h}^+_l \prod_{i=1}^l x_i^{\langle \alpha_i,\alpha_l\rangle} \, = \, 1.
\end{equation}
By choosing a logarithm, solutions $\boldsymbol{\overline{x}}=(\overline{x}_1,\dots,\overline{x}_l)$ of the above equations are obtained as exponential functions
of solutions to the linear system
\[
\Gamma \, (\log(x_1), \ldots, \log(x_l))^{tr} \, \, = \, \, -(\log(\overline{h}^+_1), \ldots, \log(\overline{h}^+_l))^{tr}\,,
\]
where $\Gamma$ is the Cartan matrix of $\liealg$. Hence, each $\overline{x}_i$ is a certain product of powers
of $\overline{h}^+_1$, \ldots, $\overline{h}^+_l$ with rational exponents $q_i$. 
Let $\algext=\resfield(\boldsymbol{\overline{x}})$
be the algebraic extension of $\resfield $ generated by $\boldsymbol{\overline{x}}$ (whose structure as a differential field is uniquely 
determined). For the algebraic extension $\algext_1 = \field\langle \bbi_w \rangle \langle \bhp\rangle (\boldsymbol{\overline{x}})$ we have 
\[
\algext_1\langle \bhm,\bhz \rangle \, = \, \algext\,.
\]
Since $\bhm,\bhz$ and $\bhp$ are differentially algebraically independent over $\field\langle \bbi_w \rangle$ 
by Lemma~\ref{TransStep1}  and since $\algext_1$ is defined over $ \field\langle \bbi_w \rangle \langle \bhp\rangle$, we conclude that 
$\bhm,\bhz$ are differentially algebraically independent over $\algext_1$. Let $\bgp,\bgm,\bgz$ be the coefficients of the linear 
representation of 
$$\gauge{t}{A(\bhp,\bhm,\bhz)} \, = \, A(\bgp,\bgm,\bgz),$$
where $t=\boldsymbol{t}(\boldsymbol{\overline{x}})$.
As above, for $\alpha \in \Phi$ the effect of the adjoint action of $t_i(\overline{x}_i)$  on $X_{\alpha}$ is multiplication by 
the scalar $\overline{x}_i^{\langle \alpha_i,\alpha \rangle}$ and the action of $\Ad (\boldsymbol{t}(\boldsymbol{\overline{x}}))$ on the Cartan subalgebra is trivial. 
Further the logarithmic derivative of $\boldsymbol{t}(\boldsymbol{\overline{x}})$ is an element of $\mathfrak{h}$. We conclude that, since $\boldsymbol{\overline{x}} \in \algext_1$, the elements $\overline{g}^0_i$ and $\overline{h}^0_i$ differ by an element of $\algext_1$ and that $\overline{g}^+_j$, $\overline{g}^-_j$ are obtained from $\overline{h}^+_j$, $\overline{h}^-_j$ by multiplication by an element of $\algext_1$.
Then, by the choice of $\boldsymbol{\overline{x}}$, the first $l$ entries of $\bgp$ are equal to $1$ and the others remain zero. We conclude further that
$\bhm$, $\bhz$ are contained in $\algext_1\langle \bgm,\bgz \rangle $ and so we have 
$$\algext_1\langle \bgm,\bgz \rangle \, = \, \algext.$$
It follows that $ \bgm,\bgz$ are differentially algebraically independent over $\algext_1$, since the same holds for $\bhm$, $\bhz$ over $\algext_1$
(due to the invariance of the differential transcendence degree).
\end{proof}

\begin{lemma}\label{TransStep3}
There exists $\bt =(\overline{t}_1,\dots,\overline{t}_l) \in \algext_1 \{ \bgm,\bgz \}^l $ with the following properties:
\begin{enumerate}
    \item \label{claim_transstep3_a} The matrix $A(\bgp,\bgm,\bgz)$ is gauge equivalent to $A_{\group}(\bt)$ over $\algext$.
    \item \label{claim_transstep3_b} Let $i_1,\dots,i_l$ be the indices of the complementary roots $-\alpha_{i_1},\dots,-\alpha_{i_l}$.
    Then, for $j=1,\dots,l$, the differential polynomial $\overline{t}_j$ does not involve
    the indeterminates $\overline{g}^-_{i_{j+1}},\dots, \overline{g}^-_{i_l}$ nor any proper derivative of $\overline{g}^-_{i_j}$,
    but the degree of $\overline{t}_j$ in the indeterminate $\overline{g}^-_{i_j}$ is equal to one,
    and the coefficient of $\overline{g}^-_{i_j}$ in $\overline{t}_j$ is constant.
\item \label{claim_transstep3_c} Let $\algext_2$ be the differential field generated by $\bgz$ and 
$$\{ \overline{g}^-_i \in \bgm \mid i \notin \{ i_1,\dots,i_l \}  \}  $$
over $\algext_1$.
Then $\algext_2\langle \bt \rangle =\algext$
and the elements $\bt$ are differentially algebraically independent over $\algext_2$.
\end{enumerate}
\end{lemma}

\begin{proof}
It follows from \cite[Lemma~6.8]{Seiss} (Transformation Lemma) that the matrix 
$A(\bgp,\bgm,\bgz)$ is gauge equivalent 
by a unipotent matrix $u \in \unipotent^-(\algext_1 \{ \bgm,\bgz \})$
to a matrix  $A_{\group}(\bt)$ with $\bt=(\overline{t}_1,\dots,\overline{t}_l)$ in 
$\algext_1 \langle \bgm,\bgz \rangle^l $. More precisely, since  
$u$ is unipotent with entries in $\algext_1 \{ \bgm,\bgz \}$, the inverse of $u$ also has entries in 
$\algext_1 \{ \bgm,\bgz \}$ and so we have that 
$\overline{t}_i \in \algext_1 \{ \bgm,\bgz\}$ for all 
$1 \leq i \leq l$. This shows \ref{claim_transstep3_a}.

In the proof of the Transformation Lemma we gauge transform $A(\bgp,\bgm,\bgz)$ successively with products of root group elements corresponding to roots of height $-1,\dots,-m$. The effect on $A(\bgp,\bgm,\bgz)$ is that we successively eliminate all entries corresponding 
to roots of height $0,-1,-2,\dots,-(m-1)$ except the complementary roots.
In each step the parametrization of the product of root group elements of height $-k$ depends linearly on the entries which we are going to eliminate, i.e. the entries corresponding to the roots of height $-(k-1)$. The only non-zero entries in $A(\bgp,\bgm,\bgz)$ which  correspond to positive roots are the entries belonging to the simple roots. In each step the effect of the adjoint action coming from the positive simple roots is only the elimination described above and all other contributions of the adjoint action stem from entries belonging to the negative roots.
One proves now inductively on $-k=-1,\dots,-m$ the following claim: After performing step $-k$ an entry corresponding to a root of height 
$h$ does not involve indeterminates among the $\bgm$ corresponding to roots of height less than $h$ nor a proper derivative of 
$\overline{g}^-_h$, but the degree of this entry in $\overline{g}^-_h$ is equal to one and the coefficient is constant.
In each step of the induction one concludes first with Remark~\ref{remark3} and \ref{remark4} that to entries of height 
less than or equal to $-k$ only differential polynomials in indeterminates $\bgz$ and those of $\bgm$ which correspond to roots of height greater than $-k$ are added. Secondly one concludes with Remark~\ref{remark3} and \ref{remark4} that to entries of height 
greater than $-k$ nothing is added.
Since the complementary roots, which among themselves are ordered by height, are the only ones which are not eliminated we conclude statement~\ref{claim_transstep3_b}.
 
Next we prove that $\algext_2\langle \bt \rangle = \algext_1 \langle \bgz,\bgm \rangle = \algext$. We need to show 
that $\algext_2\langle \bt  \rangle$ contains all $\overline{g}^-_{i_1}, \dots, \overline{g}^-_{i_l}$. 
By statement~\ref{claim_transstep3_b} we can successively express the elements $\overline{g}^-_{i_1},\dots,\overline{g}^-_{i_l}$ in terms of $\bt$ and 
$$ \{ \bgz,\bgm \} \setminus \{ \overline{g}^-_{i_1},\dots,\overline{g}^-_{i_l} \}  $$
over $\algext_1$. Indeed, $\overline{t}_1$ depends linearly on $\overline{g}^-_{i_1}$ and does not involve any other indeterminate 
$\overline{g}^-_{i_2},\dots,\overline{g}^-_{i_l}$ nor any proper derivative of $\overline{g}^-_{i_1}$ and so we can express 
$\overline{g}^-_{i_1}$ as stated. The expression $\overline{t}_2$ does not involve any of the indeterminates $\overline{g}^-_{i_3},\dots,\overline{g}^-_{i_l}$ nor a 
derivative of $\overline{g}^-_{i_2}$ and $\overline{g}^-_{i_2}$ appears linearly in it. Using the expression of 
$\overline{g}^-_{i_1}$ obtained before we can write $\overline{g}^-_{i_2}$ as stated. 
Continuing in this way we get that all $\overline{g}^-_{i_1},\dots,\overline{g}^-_{i_l}$ are contained in $\algext_2\langle \bt  \rangle$.

Since $\bgz$, $\bgm$ are differentially algebraically independent over $\algext_1$ by Lemma~\ref{TransStep2} and we already proved that
$$ \algext_2\langle \bt \rangle \, = \, \algext_1\langle \bgz \rangle \langle \overline{g}^-_i \in \bgm \mid i \notin \{ i_1,\dots,i_l \}  \rangle \langle \bt \rangle   \, = \, \algext_1 \langle \bgz,\bgm \rangle\,, $$
it follows from the invariance of 
the differential transcendence degree that $\bt$ are differentially algebraically independent over $\algext_2$.
\end{proof}

In the proof of the following theorem we use Theorem~\ref{thm:difftransfullgroup} which will be presented and proved in the next section.
\begin{theorem}\label{thm:gauge_eq_&_full_group}
There is a differential extension field $\algext$ of 
$\field\langle \bap, \bam,\baz  \rangle$ with field of constants $\field$ 
such that 
the matrix $A(\bap, \bam,\baz )$ is 
gauge equivalent to $A_{\group}(\bt)$ over $\algext$ for some 
$\bt=(\overline{t}_1,\dots,\overline{t}_l)$ and the differential 
Galois group of $A_{\group}(\bt)$ over $\algext$ is the full group $\group(\field)$. 
\end{theorem}
\begin{proof}
We obtain from Lemma~\ref{TransStep1} that the matrix 
$A( \bap, \bam,\baz)$ is gauge equivalent to 
$A(\bhp,\bhm,\bhz)$ over $\resfield$ where
$\bhp=(\overline{h}^+_1,\dots,\overline{h}^+_l,0,\dots,0)$ and the field of constants of $\resfield$ is $\field$ by 
Lemma~\ref{lem:fieldofconstants}. Next  
Lemma~\ref{TransStep2} implies that there is an algebraic extension $\algext$
of $\resfield$ such that $A(\bhp,\bhm,\bhz)$ is 
gauge equivalent to $A(\bgp,\bgm,\bgz)$ 
over $\algext$ where $\bgp=( 1,\dots,1,0,\dots,0)$.
Since $\field$ is algebraically closed the field of constants of $\algext$ is also $\field$.
Finally Lemma~\ref{TransStep3} implies that 
$A(\bgp,\bgm,\bgz)$  
is gauge equivalent to the matrix $A_{\group}(\bt)$ over $\algext$ and that there is a subfield 
$\algext_2$ of $\algext$ such that $\algext_2\langle \bt \rangle=\algext$ and 
the elements $\bt$ are differentially algebraically independent over $\algext_2$. 
Applying Theorem~\ref{thm:difftransfullgroup} to $A_{\group}(\bt)$ and 
$\algext_2\langle \bt \rangle=\algext$ we obtain that the differential Galois group of 
$A_{\group}(\bt)$ over $\algext$ is $\group(\field)$.
\end{proof}
\begin{remark}
Theorem~\ref{thm:gauge_eq_&_full_group} gives an alternative proof for the fact that the Galois group of 
$A(\bap, \bam,\baz )$ over 
$\field \langle \bap, \bam,\baz \rangle $ is $\group(\field)$ which was already proven by L.~Juan in \cite{Juan}.
Indeed, since $A(\bap, \bam,\baz )$ is gauge equivalent to 
$A_{\group}(\bt)$ over $\algext$, the Galois group of 
$A(\bap, \bam,\baz )$ over $\algext$ is also $\group(\field)$. Because 
$\group$ is an upper bound for the Galois group and $A(\bap, \bam,\baz)$ has 
entries in $\field \langle \bap, \bam,\baz  \rangle $, the Galois group of 
$A(\bap, \bam,\baz )$ over 
$\field \langle \bap, \bam,\baz  \rangle $ is $\group(\field)$ as desired.
\end{remark}

\section{The Galois Group of the Normal Form over $F\langle \bt \rangle$}\label{sec:NF_Galois}

Let $\difffield$ be an arbitrary differential field with field of constants $\field$. For differential indeterminates 
$\bt=(\overline{t}_1,\dots,\overline{t}_l)$ over $F$ we prove in this section that the Galois group of $A_{\group}(\bt)$ over $\difffield\langle \bt \rangle$ is $\group(\field)$ by adapting to $\difffield$ the construction presented in \cite{Seiss} over $\field$.
To this end let $\bv=(v_1,\dots,v_l)$ be differential indeterminates over $\difffield$. Then $\bv$ are differentially algebraically independent over $\field$.
The results of \cite{Seiss} can be summarized in the following proposition. 
\begin{proposition}\label{prop:generic_extension}
Let $n(\lweyl)$ be a representative of the longest Weyl group element in the normalizer of the torus. Then  
there are non-zero $\bc =(c_1,\dots,c_l)$ in $\field$ and polynomials 
$g_1(\bv) ,\dots, g_l(\bv) \in \field[\bv]$ as well as differential polynomials $f_{l+1}(\bv),\dots , f_m(\bv)$ and 
$s_{1}(\bv), \dots, s_{l}(\bv)$ in $\field\{ \bv \}$ having the following properties:
\begin{enumerate}
\item\label{propext:a}
The polynomials $g_1(\bv),\dots, g_l(\bv)$ are $\field$-linearly independent and homogeneous of degree one.
\item\label{propext:b}
Let $I_{\bs} \subset \field\{ \bv \}$ be the differential ideal generated by $\bs=(s_{1}(\bv), \dots, s_{l}(\bv))$.
Then the differential ring $\field\{ \bv \}/ I_{\bs} $ is isomorphic to a polynomial ring in finitely many variables which can be chosen to be the images of $v_1,\dots,v_l$ and finitely many 
of their derivatives.
The differential polynomials $\bs$ are differentially algebraically independent over $\field$.
\item\label{propext:c}
The matrix 
\begin{equation*}
\ALiou(\bv) \, = \, \sum_{i=1}^l g_i(\bv) H_i + \sum_{i=1}^l c_i X_{-\alpha_i}  
 \in \mathfrak{b}^-(\field\langle \bv \rangle)
\end{equation*}
defines a Picard-Vessiot extension $\generalext$ of $\field\langle \bv \rangle$ with differential Galois group $\borel^-(\field)$ and has a fundamental matrix 
$\YLiou \in \borel^-(\generalext)$. 
\item\label{propext:d}
The logarithmic derivative of 
$$Y=\boldsymbol{u}\big( \bv,f_{l+1}(\bv),\dots, f_m(\bv)\big) \,  n(\lweyl) \, \YLiou$$
is $A_{\group}(\bs)$ and $\generalext = \field\langle \bv \rangle(\YLiou)$ is a Picard-Vessiot extension of $\field\langle \bs \rangle$ for $A_{\group}(\bs)$ 
with differential Galois group  $\group(\field)$.
 \end{enumerate}
\end{proposition}
\begin{proof}
\begin{enumerate}
    \item The $g_i(\bv)$ correspond to the $\overline{g}_i(\boldsymbol{\eta})$ defined on page~16 in \cite{Seiss} and the desired properties follow from Lemma~5.3 (a) of \cite{Seiss}.
    \item The $s_i(\bv)$ correspond to the $h_{j_i}(\boldsymbol{\eta})$ defined in Lemma~5.11 of \cite{Seiss} and so Theorem~5.13 of \cite{Seiss} yields that $s_1(\bv),\dots, s_l(\bv)$ are differentially algebraically independent over $\field$. 
    
    The properties of $h_{j_1}(\boldsymbol{\eta}),\dots,h_{j_l}(\boldsymbol{\eta})$ stated in \cite[Lemma~5.11]{Seiss} imply 
    that the differential ideal $I_{\bs}$ is also generated by $l$ differential polynomials $\tilde{s}_1(\bv),\dots,\tilde{s}_l(\bv)$ with the following properties:
    \begin{enumerate}
        \item Each $\tilde{s}_i(\bv)$ can be solved for some $v_{k_i}^{(d_i)}$, where $d_i \in \Z_{\geq 0}$.
        \item The differential indeterminate $v_{k_i}$ appears in the remaining terms of $\tilde{s}_i(\bv)$ with differentiation order lower than $d_i$. 
        \item The differential indeterminates $v_{k_1},\dots ,v_{k_l}$ are pairwise distinct.
    \end{enumerate}
    We conclude that $\field\{ \bv \}/ I_{\bs}$ is isomorphic to a polynomial ring in finitely many variables, which we choose to be the images of 
    the indeterminates $v_{k_i},v_{k_i}',\dots, v_{k_i}^{(d_i-1)}$ with $1\leq i \leq l$.
    \item This corresponds to Proposition 5.5 (a) and (b) in \cite{Seiss}.
    \item This corresponds to Theorem~5.13 in \cite{Seiss}.
    \endproof
\end{enumerate}  
\end{proof}

\begin{corollary} $\,$ \label{cor:properties_of_s}
We continue to use the notation of Proposition \ref{prop:generic_extension}.
\begin{enumerate}
\item\label{corprop:a}
Let $I_{\bs} \subset \difffield\{ \bv \}$ be the differential ideal generated by $\bs$.
Then the differential ring 
$\difffield\{ \bv \}/ I_{\bs} $ is isomorphic to a polynomial ring in $m$ indeterminates which can be chosen to be the images of 
$v_1,\dots,v_l$ and finitely many of their derivatives. 
\item\label{corprop:b}
The differential polynomials $\bs$ are differentially algebraically independent over $\difffield$.
\end{enumerate}
\end{corollary}
\begin{proof}
By Proposition~\ref{prop:generic_extension} \ref{propext:b} the quotient  
$\field\{ \bv \}/I_{\bs}$ is isomorphic to a polynomial ring $\field[\bz]$, where the indeterminates $\bz=(z_1,\dots,z_k)$ are the images of $\bv$ and finitely many of their
derivatives, $k = d_1 + d_2 + \ldots + d_l$.
Since $\bv$ are differential indeterminates over $\difffield$, the elements of $$\{ v_i^{(j)} \mid 1\leq i \leq l, \, j \in \Z_{\geq 0} \}$$
are algebraically independent over $\difffield$. The same reasoning as in the proof of Proposition~\ref{prop:generic_extension} \ref{propext:b} shows that the images $\bz$ in $\difffield\{ \bv \}/I_{\bs}$ of the same $v_{k_i}, v_{k_i}',\dots, v_{k_i}^{(d_i-1)}$ with $1\leq i \leq l$ are algebraically independent over $\difffield$. The algebraic independence of $\bv$ and $\bz$ implies the isomorphisms 
$$
\difffield\{ \bv \}/I_{\bs} \, \cong \, \field\{ \bv \}/I_{\bs} \otimes_{\field} \difffield \, \cong \, \field[\bz] \otimes_{\field} \difffield \, \cong \,  \difffield[\bz] .$$

We prove that $k=m$. Again by Proposition~\ref{prop:generic_extension} \ref{propext:b} the transcendence degree of $\field\langle \bv \rangle$ over $\field\langle \bs \rangle$ is $k$. From Proposition~\ref{prop:generic_extension} \ref{propext:c} and \ref{propext:d} we obtain
$$\mathrm{trdeg}_{\field\langle \bv \rangle}(\generalext) = \dim(\borel^-)=l+m$$ and
$$\mathrm{trdeg}_{\field\langle \bs \rangle}(\generalext) = \dim(\group)=l+2m.$$
We conclude that $\mathrm{trdeg}_{\field\langle \bs \rangle}(\field\langle \bv \rangle)=m$ and so $k=m$. This proves \ref{corprop:a}.

Since the transcendence degree of $\difffield\langle \bv \rangle$ over $\difffield\langle \bs \rangle$ is finite, the differential transcendence degree of $\difffield\langle \bv \rangle$ over $\difffield\langle \bs \rangle$ is zero and so 
the differential transcendence degree of $\difffield\langle \bs \rangle$ over $\difffield$ is $l$ by \cite[II.9, Corollary~2]{Kolchin}. As $\difffield\langle \bs \rangle$ is generated by the $l$ elements $\bs$ over $\difffield$, they are differentially algebraically independent over $\difffield$.
\end{proof}

\begin{proposition} \label{prop:Liouville_part}
There is a Picard-Vessiot extension $\widehat{\generalext}$ of $\difffield \langle \bv \rangle$ for 
$\ALiou(\bv)$ with fundamental matrix $\calYLiou \in \borel^-(\widehat{\generalext})$. The differential Galois group of 
$\widehat{\generalext}$ over $\difffield \langle \bv \rangle$ is $\borel^-(\field)$.
\end{proposition}
\begin{proof}
First we construct a Picard-Vessiot ring for $\ALiou(\bv)$ in the usual way.
We extend the derivation of $\difffield\langle \bv \rangle $ to the ring
$$\difffield\langle \bv \rangle [\GL_n]=\difffield\langle \bv \rangle[X,\det(X)^{-1}]$$
by defining $X'=\ALiou(\bv)X$, where $X=(X_{ij})$ is a matrix of $n^2$ indeterminates $X_{ij}$. Since $\ALiou(\bv) \in \mathfrak{b}^-(\difffield\langle \bv \rangle)$, the embedding $(I)$ of the defining ideal 
$I\subset \field[\GL_n]$ of the Borel group $\borel^-$ is a differential ideal in $\difffield\langle \bv \rangle [\GL_n]$. Let $J \subset \difffield\langle \bv \rangle [\GL_n]$ be a  maximal differential ideal with $(I)\subseteq J$.
Then the field of fractions $\widehat{\generalext}$
of $\difffield\langle \bv \rangle[\GL_n] / J$
is a Picard-Vessiot extension of $\difffield\langle \bv \rangle$ and the matrix of residue classes $\calYLiou:=\overline{X}$ is a fundamental matrix contained in 
$\borel^-(\widehat{\generalext})$.

It is left to show that the Galois group of $\widehat{\generalext}$ over $\difffield\langle \bv \rangle$ is $\borel^-(\field)$. To this end we apply \cite[Theorem 2]{KovacicInvProb} to $\ALiou(\bv)$. Clearly we have the semi-direct product $B^-=\torus \, \unipotent^-$. The roots and root groups
of the commutative unipotent group $\unipotent^-/[\unipotent^-,\unipotent^-]$ with respect to $\torus$ correspond to the negative simple roots $-\alpha_1,\dots,-\alpha_l$ and the root groups $U_{-\alpha_1},\dots,U_{-\alpha_l}$ of $\unipotent^-$. Hence, the dimension of each root group of $\unipotent^-/[\unipotent^-,\unipotent^-]$ is one. Let $\chi_1,\dots,\chi_l$
be the generators of the character group of $T$ and denote by $e_{ij} \in \Z$ the exponent of $\chi_j$ when writing $-\alpha_i$ as a 
product of $\chi_1,\dots,\chi_l$. As in \cite[p.~596]{KovacicInvProb} we define the map 
$$ L_{g_i(\bv),-\alpha_i}: \difffield\langle \bv \rangle \longrightarrow \difffield\langle \bv \rangle, \quad x \longmapsto x' - \sum_{j=1}^l e_{ij} g_j(\bv) x, $$
where $g_1(\bv)$, \ldots, $g_l(\bv)$ have been defined in Proposition~\ref{prop:generic_extension}, namely by
$$\ALiou(\bv) \, = \, \sum_{i=1}^l g_i(\bv) H_i + \sum_{i=1}^l c_i X_{-\alpha_i}. $$
Now Theorem~2 of \cite{KovacicInvProb} implies that the differential Galois group of $\ALiou(\bv)$
is $\borel^-(\field)$, if the images of $g_1(\bv),\dots,g_l(\bv)$ in 
the quotient $\difffield\langle \bv \rangle/\dlog (\difffield\langle \bv \rangle)$ are $\Z$-linearly independent, i.e.\ form a $\Z$-submodule of rank $l$ and if $c_i$ does not 
reduce to zero in the quotient $\difffield\langle \bv \rangle/  L_{g_i(\bv),-\alpha_i} (\difffield\langle \bv \rangle)$ for all $1\leq i \leq l$, i.e., the image of every $c_i$ generates a $\field$-vector space of dimension one.
By Proposition~\ref{prop:generic_extension} the elements $\bc$ are non-zero constants and $g_1(\bv),\dots,g_l(\bv)$ are $\field$-linearly independent homogeneous polynomials of degree one. We conclude that $\bc$ and $g_1(\bv),\dots,g_l(\bv)$ satisfy the conditions of the above mentioned theorem and so the differential Galois group of $\widehat{\generalext}$ over $\difffield\langle \bv \rangle$ is $\borel^-(\field)$.   
\end{proof}

\begin{theorem}\label{thm:difftransfullgroup}
Let $\difffield$ be a differential field with field of constants $\field$ and let $\bt=(t_1,\dots,t_l)$ be differential indeterminates
over $\difffield$. Then the differential Galois group of $A_{\group}(\bt)$ over $\difffield\langle \bt \rangle$ is $\group(\field)$.
\end{theorem}
\begin{proof}
 By Proposition~\ref{prop:Liouville_part} the matrix $\ALiou(\bv)$ defines
 a Picard-Vessiot extension $\widehat{\generalext}$ of $\difffield\langle \bv \rangle$ with
 differential Galois group $\borel^-(\field)$ and has a fundamental matrix $\calYLiou \in \borel^-(\widehat{\generalext})$.
As a first consequence the transcendence degree of $\widehat{\generalext}$ over $\difffield \langle \bv \rangle$ is
$$\dim(\borel^-) \, = \, m+l.$$
Further, since the logarithmic derivatives of $\calYLiou$ and $\YLiou$ are equal, Proposition~\ref{prop:generic_extension} \ref{propext:d} implies that the 
logarithmic derivative of
$$\mathcal{Y} \, = \, \boldsymbol{u}\big( \bv,f_{l+1}(\bv),\dots, f_m(\bv)\big) \,  n(\lweyl) \, \calYLiou $$
is $A_{\group}(\bs)$. By Corollary~\ref{cor:properties_of_s} \ref{corprop:b} the differential polynomials $\bs$ are differentially algebraically independent over $\difffield$.
Moreover, $\widehat{\generalext}$ is a Picard-Vessiot extension of $F\langle \bs \rangle$ for $A_{\group}(\bs)$. Indeed, the fields of constants of $\widehat{\generalext}$ and $\difffield\langle \bs \rangle$ are $\field$, since $\widehat{\generalext}$ is a Picard-Vessiot extension and $\difffield\langle \bs \rangle$ is contained in $\difffield\langle \bv \rangle$. 
The matrix $\mathcal{Y}$ is a fundamental matrix for $A_{\group}(\bs)$ and its entries clearly lie in $\widehat{\generalext}$. In order to prove that $\mathcal{Y}_{ij}$ generate $\widehat{\generalext} = \difffield\langle \bv \rangle(\calYLiou)$ as a field over $F\langle \bs \rangle$, we show the non-trivial inclusion $\widehat{\generalext} \subseteq F \langle \bs \rangle(\mathcal{Y}_{ij})$.
Since $\mathcal{Y}\in \group(\widehat{\generalext})$, the Bruhat decomposition and its uniqueness imply that we can express $\bv$, 
$f_{l+1}(\bv),\dots, f_m(\bv) $ and the entries of $\calYLiou$ as rational functions in $\mathcal{Y}_{ij}$. Since  $F \langle \bs \rangle(\mathcal{Y}_{ij})$ is a differential field, the derivatives of $\mathcal{Y}_{ij}$ being defined by $A_{\group}(\bs)$, it contains besides the entries of $\calYLiou$ all derivatives of $\bv$ and so  $\widehat{\generalext} \subseteq F \langle \bs \rangle(\mathcal{Y}_{ij})$. Hence, 
 $\widehat{\generalext} = \difffield\langle \bs \rangle(\mathcal{Y}_{ij})$ and $\widehat{\generalext}$ is a Picard-Vessiot extension of $\difffield\langle \bs \rangle$ for $A_{\group}(\bs)$.
\begin{center}
\begin{tikzpicture}[thick,scale=0.75]
\node[circle, draw, fill=black, inner sep=0pt, minimum width=4pt] (A) at (0,2) {};
\node[circle, draw, fill=black, inner sep=0pt, minimum width=4pt] (B) at (0,1) {};
\node[circle, draw, fill=black, inner sep=0pt, minimum width=4pt] (C) at (0,0) {};
\node (D) at (0,1.5) {};
\node (E) at (0,0.5) {};
\draw (A) -- (B);
\draw (B) -- (C);
\node (LA) [left=1em of A] {$\widehat{\generalext}$};
\node (LB) [left=1em of B] {$\difffield\langle \bv \rangle$};
\node (LC) [left=1em of C] {$\difffield\langle \bs \rangle$};
\node (LD) [right=1em of D] {${\rm trdeg} = m+l$};
\node (LE) [right=1em of E] {${\rm trdeg} = m$};
\end{tikzpicture}
\end{center}
By Corollary~\ref{cor:properties_of_s} \ref{corprop:a} the transcendence degree of $\difffield\langle \bv \rangle$ over $\difffield\langle \bs \rangle$ is $m$ and so the transcendence degree of $\widehat{\generalext}$ over $\difffield\langle \bs \rangle$ is 
$$2m+l \, = \, \mathrm{dim}(\group). $$
We conclude that the differential Galois group of $\widehat{\generalext}$ over $\difffield\langle \bs \rangle$ for $A_{\group}(\bs)$ is $\group(\field)$. 
Since $\bs$ are 
differentially algebraically independent over $\difffield$, the
Galois group of $A_{\group}(\bt)$ over $\difffield\langle \bt \rangle$ is $\group(\field)$, where $\bt$ are differential indeterminates over $\difffield$.
\end{proof}

\section{Specializing the Coefficients of the Generic Equation}\label{sec:specializing}
Let $\algext$ be the extension field of $\field\langle \bap,\bam, \baz \rangle$ constructed in the previous sections with respect to the resolving Weyl group element $\lweyl$, i.e.~the longest Weyl group element. Note that in this case we have  
$$\Psi \, = \, \{ -\alpha_{l+1},\dots,-\alpha_m \}\,, \quad  \bb_{\lweyl} \, = \, (b_{l+1},\dots,b_m)\,, $$
and that $\algext$ is an algebraic extension of the differential field  
$$ \resfield \, = \, \Frac(\diffring) \, = \, \field\langle \bbi_{\lweyl}, \bapi, \bami, \bazi \rangle, $$
where we denote by $\diffring$ the differential ring
\begin{equation}\label{defR}
\diffring \, = \, \field\{ \bb_{\lweyl}, \bap, \bam, \baz \} / \Imw \, = \, \field\{ \bbi_{\lweyl}, \bapi, \bami, \bazi \} .
\end{equation}

Let $\multset$ be the multiplicatively closed subset generated by $\overline{h}^+_1,\dots,\overline{h}^+_l$ and let $\boldsymbol{\overline{x}}=(\overline{x}_1,\dots,\overline{x}_l)$ be 
the solutions of \eqref{eq:normalization} introduced in the proof of Lemma~\ref{TransStep2}, i.e., each $\overline{x}_i$ is a certain product of powers of $\overline{h}^+_1,\dots,\overline{h}^+_l$ with rational exponents. Then 
$$ \locdiffring \, := \, (\multset^{-1}\diffring)[\boldsymbol{\overline{x}}]  $$
defines an integral ring extension $\locdiffring$ of $\multset^{-1}\diffring$ and the gauge transformation of $A(\bapi, \bami, \bazi )$ to $A_{\group}(\bt)$ as in Theorem~\ref{thm:gauge_eq_&_full_group} is defined over $\locdiffring$. Indeed, following the proof of the theorem, the matrix $A(\bapi, \bami, \bazi )$ is gauge equivalent to 
$A(\bhp,\bhm,\bhz )$ over $\diffring$, where $\bhp=(\overline{h}^+_1,\dots,\overline{h}^+_l,0,\dots,0) \in \diffring^m$ with non-zero $\overline{h}^+_1,\dots,\overline{h}^+_l$. Since $\locdiffring$ contains the roots for the normalization of the coefficients $\overline{h}^+_1,\dots,\overline{h}^+_l$, it follows that $A(\bhp,\bhm,\bhz )$ is gauge equivalent to
$A(\bgp,\bgm,\bgz)$ with 
$\bgp=(1,\dots,1,0,\dots,0)$. The transformation from $A(\bgp,\bgm,\bgz)$ to $A_{\group}(\bt)$ needs no further inverse elements nor algebraic extensions and is therefore also defined over $\locdiffring$.

Let $\difffield$ be an arbitrary differential field with field of constants $\field$ and let   
$$A \, = \, A(\bsap, \bsam, \bsaz) \in \liealg(\difffield)$$
define a Picard-Vessiot extension $\extfield$ with Galois group $\group(\field)$.
For differential indeterminates $\bsb=((b_*)_{l+1},\dots,(b_*)_m)$ over $\extfield$ we denote by
$\sigma$ the homomorphism of differential rings 
$$ \sigma: \field\{\bb_{\lweyl}, \bap, \bam,\baz \} \longrightarrow \extfield\{ \bsb \}, \quad (\bb_{\lweyl},\bap, \bam,\baz)\longmapsto (\bsb,\bsap, \bsam,\bsaz).$$
We fix this notation for the rest of the section.
\begin{theorem}\label{thm:specialization}
Suppose $A=A(\bsap, \bsam, \bsaz)$ defines a Picard-Vessiot extension $\extfield$ of $\difffield$ with differential Galois group $\group(\field)$. 
If $\sigma(\Imw)$ is a maximal differential ideal in $\extfield\{ \bsb \}$, then there exists a 
differential extension field $\nongenext$ of $\difffield$ with the following properties:
\begin{enumerate}
    \item The field of constants of $\nongenext$ is $\field$.
    \item There exists a differential homomorphism $\overline{\sigma}: \locdiffring \rightarrow \nongenext$ such that 
    $$\overline{\sigma}(A(\bapi, \bami, \bazi )) \, = \, A.$$  
    \item The matrix $A$ is gauge equivalent over $\nongenext$ to the specialization $\overline{\sigma}(A_{\group}(\bt))$ of $A_{\group}(\bt)$.
    \item The differential Galois group of $A$ over $\nongenext$ is $\group(\field)$.
\end{enumerate}
\end{theorem}
\begin{proof}
First we construct a purely transcendental differential field extension $\difffield\langle \bsbi \rangle$ of $\difffield$ with constants $\field$
and show that the Galois group over $\difffield\langle \bsbi \rangle$ is still the full group $\group$.
Since $\sigma(\Imw)$ is a maximal differential ideal in $\extfield\{ \bsb \}$, the field of constants of 
$$\extfield\langle \bsbi \rangle \, = \, \Frac(\extfield\{ \bsb \}/\sigma(\Imw)) $$
is $\field$. Thus the fields of constants of the subring 
$$\extfield\{ \bsbi \} \, = \, \extfield\{ \bsb \}/\sigma(\Imw) $$
and of the subfield $\difffield\langle \bsbi \rangle$ are also $\field$.
Hence, the differential field extension $\extfield\langle \bsbi \rangle$ of $\difffield\langle \bsbi \rangle$ is a Picard-Vessiot extension for $A$. Since $A\in \liealg(\difffield)$, its differential Galois group $H(\field)$ is a subgroup of $\group(\field)$. We prove that $H(\field)=\group(\field)$. 
Recall from Section~\ref{sec:res_weyl_diff_systems} that $\Imw$ is generated by $f^+_{\lweyl,l+1},\dots, f^+_{\lweyl,m}$
and by Theorem~\ref{thm:Sw_simple} the leader of $f^+_{\lweyl,i}$
is $b_i'$ and $f^+_{\lweyl,i}$ is linear in its leader and has constant initial. 
Then the same is true for the generators $\sigma(f^+_{\lweyl,l+1}),\dots, \sigma(f^+_{\lweyl,m})$ of $\sigma(\Imw)$ and, since $\sigma(\Imw)$ is a maximal differential ideal, the transcendence degree of $E\langle \bsbi \rangle$ over $\extfield$ is 
$m-l$. Moreover, the assumption on the Galois group implies
$$\mathrm{trdeg}_{\difffield}(\extfield) \, = \, \dim(\group) \, = \, 2m+l$$
and so the transcendence degree of $E\langle \bsbi \rangle$ over $\difffield$ is $3m$. Again by the properties of $\sigma(f^+_{\lweyl,l+1}),\dots, \sigma(f^+_{\lweyl,m})$ and the maximality of $\sigma(\Imw)$ the transcendence degree of $\difffield\langle  \bsbi \rangle$ over $\difffield$ is also $m-l$. It follows that
\[
\dim (H(\field)) \, = \, \mathrm{trdeg}_{\difffield\langle \bsbi \rangle}(\extfield\langle \bsbi \rangle) \, = \, 2m+l \, = \, \dim (\group(\field))
\]
and thus, by the main theorem of Galois theory, we conclude that $H(\field)=\group(\field)$.
\begin{center}
\begin{tikzpicture}[thick,scale=0.75]
\node[circle, draw, fill=black, inner sep=0pt, minimum width=4pt] (A) at (0,2) {};
\node[circle, draw, fill=black, inner sep=0pt, minimum width=4pt] (B) at (0,1) {};
\node[circle, draw, fill=black, inner sep=0pt, minimum width=4pt] (C) at (0,0) {};
\node (D) at (0,1.5) {};
\node (E) at (0,0.5) {};
\draw (A) -- (B);
\draw (B) -- (C);
\node (LA) [left=1em of A] {$\extfield\langle \bsbi \rangle$};
\node (LB) [left=1em of B] {$\extfield$};
\node (LC) [left=1em of C] {$\difffield$};
\node (LD) [right=1em of D] {${\rm trdeg} = m-l$};
\node (LE) [right=1em of E] {${\rm trdeg} = 2m+l$};
\end{tikzpicture}
\qquad
\begin{tikzpicture}[thick,scale=0.75]
\node[circle, draw, fill=black, inner sep=0pt, minimum width=4pt] (A) at (0,2) {};
\node[circle, draw, fill=black, inner sep=0pt, minimum width=4pt] (B) at (0,1) {};
\node[circle, draw, fill=black, inner sep=0pt, minimum width=4pt] (C) at (0,0) {};
\node (D) at (0,1.5) {};
\node (E) at (0,0.5) {};
\draw (A) -- (B);
\draw (B) -- (C);
\node (LA) [left=1em of A] {$E\langle \bsbi \rangle$};
\node (LB) [left=1em of B] {$\difffield\langle  \bsbi \rangle$};
\node (LC) [left=1em of C] {$\difffield$};
\node (LD) [right=1em of D] {${\rm trdeg} = 2m+l$};
\node (LE) [right=1em of E] {${\rm trdeg} = m-l$};
\end{tikzpicture}
\end{center}

Clearly the differential homomorphism
$$\widetilde{\sigma}: \diffring \longrightarrow \difffield\langle \bsbi \rangle, \quad (\bapi, \bami, \bazi,\bbi_{\lweyl}) \longmapsto (\bsap, \bsam, \bsaz,\bsbi),$$
where $\diffring$ is defined in \eqref{defR},
is well-defined and satisfies $\widetilde{\sigma}(A(\bami, \bazi, \bapi ))=A$. Because
$A(\bami, \bazi, \bapi )$ can be gauge transformed to $A(\bhp,\bhm,\bhz )$ over $\diffring$ and $\widetilde{\sigma}$ is a differential homomorphism, $A$ is gauge 
equivalent to $\widetilde{\sigma}(A(\bhp,\bhm,\bhz ))$ over $\difffield\langle \bsbi \rangle$. The fact that the differential Galois group of 
$\extfield\langle \bsbi \rangle$ over $\difffield\langle \bsbi \rangle$ is $\group(\field)$ implies that $\widetilde{\sigma}(\overline{h}^+_1),\dots, \widetilde{\sigma}(\overline{h}^+_l)$ are all non-zero, since otherwise we would have obtained a gauge equivalent matrix in a proper Lie subalgebra of $\liealg$.
The extension of $\widetilde{\sigma}$ to  $\locdiffring = (\multset^{-1}\diffring)[\boldsymbol{\overline{x}}]$
depends on the existence of appropriate images of $\boldsymbol{\overline{x}}$.
Since $\widetilde{\sigma}(\overline{h}^+_1),\dots, \widetilde{\sigma}(\overline{h}^+_l)$ are non-zero, the specialization of the equations in \eqref{eq:normalization} by $\widetilde{\sigma}$ has a solution in an algebraic extension field.
Hence, there exists an algebraic extension $\nongenext$ of 
$\difffield\langle \bsbi \rangle$
and a differential homomorphism
$$\overline{\sigma}: \locdiffring \longrightarrow \nongenext$$ 
extending $\widetilde{\sigma}$.
The gauge equivalence of $A(\bhp,\bhm,\bhz )$ to $A_{\group}(\bt)$ over $\locdiffring$
implies the gauge equivalence of $A$ to $\overline{\sigma}(A_{\group}(\bt))$ over $\nongenext$. Finally, since $\group(\field)$ is 
a connected group and $\nongenext$ is an algebraic extension of $\difffield\langle  \bsbi \rangle$, the differential 
Galois group for $A$ remains $\group(\field)$ over $\nongenext$.
\end{proof}

Finally, for a Picard-Vessiot extension $\extfield$ of $\difffield$ with differential Galois group $\SL_{3}(\field)$ and defining matrix 
$$A \, = \, \sum_{i=1}^3 (a_*)^-_i X_{-\alpha_i} + \sum_{i=1}^2 (a_*)^0_i H_i + \sum_{i=1}^3 (a_*)^+_i X_{\alpha_i}  $$
we show how Theorem~\ref{thm:specialization} allows to construct a differential extension field $\nongenext$ of $\difffield$
such that  the field of constants of $\nongenext$ is $\field$, $A$ is gauge equivalent to a matrix in normal form over $\nongenext$ and the differential Galois group of $A$ over $\nongenext$ is the full group $\SL_{3}(\field)$.
To this end let $\boldsymbol{r}=(r_1, r_2)$ be two differential indeterminates over $\extfield$. Then clearly $\extfield\langle \boldsymbol{r} \rangle$ is a Picard-Vessiot extension of 
$\difffield\langle \boldsymbol{r} \rangle$ with differential Galois group $\SL_3(\field)$.
Gauge transformation of $A$ by 
$u_{-\alpha_1}(r_1) u_{-\alpha_2}(r_2)$ yields a new defining matrix 
$$A_{\boldsymbol{r}}=A(\boldsymbol{a}^+_{\boldsymbol{r}},\boldsymbol{a}^-_{\boldsymbol{r}},\boldsymbol{a}^0_{\boldsymbol{r}}) \in \liesl_3(\difffield\langle \boldsymbol{r} \rangle)$$ for the extension $\extfield\langle \boldsymbol{r} \rangle$ of $\difffield\langle \boldsymbol{r} \rangle$.
As in the paragraph before Theorem~\ref{thm:specialization} with $A$, $F$ and $E$ replaced by  $A_{\boldsymbol{r}}$, $\difffield\langle \boldsymbol{r} \rangle$ and $\extfield\langle \boldsymbol{r} \rangle$, respectively, we consider the differential ring homomorphism 
$$ \sigma: \field\{\bb_{\lweyl}, \bap, \bam,\baz \} \longrightarrow \extfield\langle \boldsymbol{r} \rangle  \{ \bsb \}, \quad (\bb_{\lweyl},\bap, \bam,\baz)\longmapsto (\bsb,\boldsymbol{a}^+_{\boldsymbol{r}},\boldsymbol{a}^-_{\boldsymbol{r}},\boldsymbol{a}^0_{\boldsymbol{r}}).$$
Once we have verified for $A_{\boldsymbol{r}}$ and $\extfield\langle \boldsymbol{r} \rangle/\difffield\langle \boldsymbol{r} \rangle$ the condition on the differential ideal $\sigma(\Imw)$, Theorem~\ref{thm:specialization} yields the desired extension $L$ of $\difffield\langle \boldsymbol{r} \rangle$ and hence of $F$. 

In order to identify the ideal $\sigma(\Imw)$, we
note that for $\SL_3$ and the longest Weyl group element $\lweyl$ we have  $\Psi=\{-\alpha_3 \}$ and $\bb_{\lweyl}=(b_3)$. Gauge transformation with $u_{-\alpha_3}((b_*)_3)$ yields then the system
$$\sigma(\system_{\lweyl}): \quad \sigma(f^+_{\lweyl,3} )=0 $$
consisting of a single equation, where $\sigma(f^+_{\lweyl,3})$ is 
\begin{gather*}
  (b_*)_3' -(a_*)^+_3  (b_*)_3^2 + 
 ((a_*)^+_3  r_1 r_2  - r_2 (a_*)^+_2  - r_1 (a_*)^+_1  + (a_*)^0_1 + (a_*)^0_2 ) (b_*)_3 \,
 +  \\ r_2 (a_*)^-_1  + (a_*)^-_3  
  - (r_2 ( (a_*)^0_2-(a_*)^0_1 ) + (a_*)^-_2 ) r_1  + (r_2 (a_*)^+_2 - (a_*)^0_2 ) r_1 r_2 - r_1  r_2' .
\end{gather*}
The next lemma gives a criterion for the ideal $\sigma(\Imw) = \langle \sigma(f^+_{\lweyl,3}) \rangle$ to be a maximal differential ideal in 
$\extfield\langle \boldsymbol{r} \rangle\{ (b_*)_3 \}$.

\begin{lemma}\label{lem:maxdiffideal}
Let $\otherfield$ be a differential field of characteristic zero.
Let 
$$f(y)= y'+ c_3 y^3 + c_2 y^2 + c_1 y + c_0   \in \otherfield\{ y \}, \quad c_0 \neq 0,$$
be a differential polynomial. 
For every positive integer $m$ we define the system $\Sigma_m$ of $m$ differential equations in $m$
differential indeterminates $y_0$, \ldots, $y_{m-1}$
\[
\left\{
\begin{array}{rcl}
y_0' \!\! & \!\! = \!\! & \!\!
c_3 ((2 y_{m-2}-y_{m-1}^2)y_0-m y_{m-3}) + (c_2 y_{m-1} - m c_1) y_0 + c_0 y_1,\\
\!\! & \!\! \vdots \!\! & \!\!\\
y_k' \! & \! = \! & \! c_3 ((2 y_{m - 2}-y_{m-1}^2) y_{k} - m y_{m - 3} + k y_{k - 2} + y_{m - 1} y_{k - 1} - 2 y_{k - 2})\\[0.2em]
\!\! & \!\! \!\! & \!\! + \, ( c_2 y_{m - 1} - m c_1 + k c_1) y_{k} + (k - m - 1) c_2 y_{k - 1} + (k+1) c_0 y_{k + 1},\\
\!\! & \!\! \vdots \!\! & \!\!\\
y_{m-1}' \!\! & \!\! = \!\! & \!\!
c_3 ((3 y_{m - 2} - y_{m - 1}^2) y_{m-1} - 3 y_{m - 3})\\[0.2em]
\!\! & \!\! \!\! & \!\! + \, (c_2 y_{m - 1} - c_1) y_{m - 1} - 2 c_2 y_{m - 2} + m  c_0,
\end{array}\right.
\]
where the convention $y_m=1$, $y_{-2}=0$, $y_{-1}=0$ is used.
If for all positive integers $m$ the system has no solution in $\otherfield^m$, then the differential 
ideal $\langle f(y) \rangle$ is a maximal differential ideal in $\otherfield\{ y \}$.
\end{lemma}
\begin{proof}
Let $g(y) \in \otherfield\{ y \}$ be a differential polynomial such that 
$$\langle f(y) \rangle \subsetneq  \langle f(y),g(y) \rangle \subseteq \otherfield\{y \}.$$
We can assume without loss of generality that $g(y)$ is a polynomial in $\otherfield[y]$.
Indeed, from $f(y)$ and all its derivatives one can express all $y^{(n)}$ in terms of polynomials in $\otherfield[y]$ modulo $\langle f(y) \rangle$. Hence, we can eliminate all derivatives in $g(y)$ and obtain a 
non-zero polynomial in $\otherfield[y]$, since otherwise we would have 
$\langle f(y) \rangle = \langle f(y),g(y) \rangle$. 
So let 
$$g(y) \, = \, \sum_{k=0}^m a_k y^k \quad \mathrm{with} \quad a_m=1 .$$
We prove by induction on the degree $m \geq 1$ that  $\langle f(y),g(y) \rangle = \otherfield\{y \}$.

Let $m=1$, i.e.\ $g(y)=y+a_0$. Differentiating $g(y)$ and subtracting $f(y)$ yields
$$\tilde{g}(y) \, = \, y' + a_0' -f(y) \, = \, a_0' -(  c_3 y^3 + c_2 y^2 + c_1 y +c_0 ).$$
In the next step we add $c_{3}y^{2}g(y)$ to $\tilde{g}(y)$ and then subtract $(c_3 a_0 - c_{2}) y g(y)$ and $(-c_3 a_0^{2} + c_{2} a_0  -  c_{1}) g(y)$ from the result. 
This leads to 
$$g_0 \, := \, a_0' +  c_3 a_0^{3} - c_{2} a_0^2 + c_{1} a_0 - c_0 \in  \langle f(y),g(y) \rangle.$$
Since by assumption the equation 
$$y_0' \, = \, -c_3 y_0^3 + c_2 y_0^2 - c_1 y_0 + c_0$$
has no solution in $\otherfield$, we obtain $g_0 \in \otherfield \setminus \{ 0\}$ and so 
$\langle f(y),g(y) \rangle = \otherfield\{y \}$.

Let $m\geq 2$. Differentiating $g(y)$ yields
\begin{gather*}
    g(y)' \, = \, m y^{m-1} y' + \big(\sum_{k=1}^{m-1} k a_k y^{k-1} y' + a_k' y^k \big) + a_0'.
\end{gather*}
We are going to reduce $g(y)'$ with differential polynomials in $\langle f(y),g(y) \rangle$ until
we obtain a polynomial of degree $m-1$ which we prove to be non-zero.
In order to eliminate the derivative $y'$ in $g(y)'$ we subtract $m y^{m-1} f(y)$ and 
$$\sum_{k=1}^{m-1} k y^{k-1} a_k f$$
from $g(y)'$ and obtain a polynomial in $\langle f(y),g(y) \rangle$ of degree $m+2$. 
More precisely, 
\begin{eqnarray*}
\tilde{g}_1(y)
\! & \! = \! & \!
-m c_3 y^{m+2} - (mc_2 + (m-1)a_{m-1} c_3) y^{m+1} -  \\
\! & \! \! & \!
(m c_1 + (m-1)a_{m-1} c_2 + (m-2) a_{m-2} c_3 ) y^m  + \\
\! & \! \! & \! \sum_{k=0}^{m-1} \big(  a_k' - (k-2) a_{k-2} c_3 - (k-1) a_{k-1} c_2 - k a_k c_1 - (k+1) a_{k+1} c_0  \big) y^k
\end{eqnarray*}
with the convention $a_{-2}= a_{-1}=0$.
In the next step we eliminate the term $-m c_3 y^{m+2}$ in $\tilde{g}_{1}(y)$. 
To this end we subtract $-m c_3 y^{2} g(y) $ from $\tilde{g}_{1}(y)$. This yields  
\begin{eqnarray*}
\tilde{g}_2(y) \! & \! = \! & \! (-m c_2 +a_{m-1} c_3  ) y^{m+1} + (-m c_1 - (m-1) a_{m-1} c_2 +2a_{m-2} c_3 ) y^m + \\
\! & \! \! & \! \sum_{k=0}^{m-1} [ a_k' + c_3 a_{m-3} m + c_0 (-k-1) a_{k+1}-c_1 k a_{k}-c_2 a_{k-1} (k-1) 
\\
\! & \! \! & \! \qquad
+ \, c_3 a_{k-2} (2-k) ] y^k .
\end{eqnarray*}
In a further elimination step we get rid of the term of degree $m+1$. We subtract from $\tilde{g}_2(y)$ 
the expression 
$(-m c_2 -(m-1)a_{m-1} c_3 + m c_3 a_{m-1} )y g(y)$. This leads to  
\begin{eqnarray*}
\tilde{g}_3(y) \! & \! = \! & \! (-c_1 m-c_3 a_{m-1}^2+c_2 a_{m-1}+2 c_3 a_{m-2} ) y^m + \\
\! & \! \! & \!
\sum_{k=0}^{m-1} \big[ a_k' + (c_3 a_{m-3}+c_2 a_{k-1}) m - c_3 a_{k-1} a_{m-1} - c_0 (k+1) a_{k+1}\\
\! & \! \! & \! \qquad
- \, c_1 k a_{k} +(-c_2 a_{k-1}-c_3 a_{k-2}) k+ c_2 a_{k-1} + 2 c_3 a_{k-2} \big] y^k .
\end{eqnarray*}
Finally, we subtract from $\tilde{g}_3(y)$ the expression 
\begin{gather*}
(-c_1 m-c_3 a_{m-1}^2+c_2 a_{m-1}+2 c_3 a_{m-2} ) g(y)
 \end{gather*}
and obtain the polynomial 
\begin{eqnarray*}
\tilde{g}_{4}(y) \! & \! = \! & \! \sum_{k=0}^{m-1} \big[ a_k' + ((a_{m-1}^2 - 2 a_{m - 2}) a_{k} - k a_{k - 2} - a_{m - 1} a_{k - 1} + a_{m - 3} m + 2 a_{k - 2}) c_3\\ 
\! & \! \! & \! \qquad
- \, ( a_{m - 1} c_2 + c_1 (k - m)) a_{k} - c_2 (k - m - 1) a_{k - 1} - c_0 (k+1) a_{k + 1} \big] y^k\\
\! & \! = \! & \! \sum_{k=0}^{m-1} \big[ a_k' - c_3(( 2 a_{m - 2}-a_{m-1}^2) a_{k} -  m  a_{m - 3} + k a_{k - 2} + a_{m - 1} a_{k - 1} - 2 a_{k - 2})\\ 
\! & \! \! & \! \qquad
- \, ( c_2  a_{m - 1} -m c_1 + c_1 k ) a_{k} - (k - m - 1) c_2 a_{k - 1} - (k+1) c_0 a_{k + 1} \big] y^k,
\end{eqnarray*}
which is of degree at most $m-1$. In order for $\tilde{g}_{4}(y)$ to be zero, all coefficients have to be zero.
But this corresponds to a solution 
$$(\overline{y}_0,\dots,\overline{y}_{m-1}) \, = \, (a_0,\dots a_{m-1}) \in \otherfield^m$$
of the system of differential equations in the statement of the lemma. Since by assumption such a solution does not exists, 
$\tilde{g}_{4}(y)$ is not zero and so by the induction hypothesis, $\langle f(y),g(y) \rangle = \langle f(y), \tilde{g}_{4}(y) \rangle = \otherfield\{y \}$.
\end{proof} 

According to the last lemma the ideal $\sigma(I_{\lweyl})=\langle  \sigma(f^+_{\lweyl,3} ) \rangle $ is 
a maximal differential ideal in $E \langle \boldsymbol{r} \rangle \{(b_*)_3\}$, if for all $m >0$ the system of differential equations in
Lemma~\ref{lem:maxdiffideal} over $E \langle \boldsymbol{r} \rangle$ has no solutions in $E \langle \boldsymbol{r} \rangle$. The remainder of this section 
is dedicated to the verification of this condition.
\begin{lemma}\label{lem:technical1}
Let $\otherfield$ be a differential field of characteristic zero and let $r$ be a differential indeterminate over $\otherfield$.
Let $a$, $b \in \otherfield \{ r\}$, $b \neq 0$, such that $$\frac{a}{b} \notin \otherfield(r^{(j)} \mid 0 \leq j \leq d-1),$$
where $d=\maxd \{\ord(a),\ord(b)\}$. 
Then $\ord(a'b-ab')=d+1$.
\end{lemma}
\begin{proof}
First assume $\ord(a)>\ord(b)$. Then the derivative of $a$ is 
$$ a'= \sep(a) r^{(d+1)} + g_a, \qquad \sep(a) \, = \, \partial_{r^{(d)}}(a), $$
where $g_a \in \otherfield[r^{(j)}\mid 0 \leq j \leq d ]$ and $\partial_{r^{(d)}}=\partial/\partial r^{(d)}$.
Since $a \notin \otherfield$, we have $\sep(a)\neq0$. By assumption $\ord(b)<d$ and so we have $\ord(ab')\leq d$. With $\ord(a')=d+1$ we conclude that $\ord(a'b)=d+1$ and so $\ord(a'b-b'a)=d+1$.

In case $\ord(a)<\ord(b)$ the claim follows as in the previous case.

Finally assume $\ord(a)=\ord(b)$. We compute 
$$a'b-ab' \, = \, r^{(d+1)}(\sep(a)b- \sep(b)a )  +  g $$
with $g \in \otherfield[r^{(j)}\mid 0 \leq j \leq d ]$.
We consider now the expression
$\sep(a)b- \sep(b)a$ as an element of the differential ring
$$ \otherfield[r^{(j)}\mid 0 \leq j \leq d-1 ][r^{(d)}] $$ with derivation $\partial_{r^{(d)}}$. 
Recognizing the coefficient of $r^{(d+1)}$ as a Wronskian determinate, we obtain that 
$$\sep(a)b- \sep(b)a \, = \, 0$$
if and only if $a,b$ are $ \otherfield(r^{(j)}\mid 0 \leq j \leq d-1 )$-linearly dependent, i.e.~$a=eb$ with 
$e \in \otherfield(r^{(j)}\mid 0 \leq j \leq d-1 ])$, which is not the case by assumption.
\end{proof}

\begin{lemma}\label{lem:technical2}
Let $a$, $b \in E [r_1,r_2] \subset E\{r_1,r_2\}$.
Then in $a'b-ab'$ the indeterminate $r_1'$ does not appear if and only if $a/b \in E(r_2)$. 
\end{lemma}
\begin{proof}
The derivatives of $a$ and $b$ are
\begin{eqnarray*}
a' \! & \! = \! & \! r_1' \partial_{r_1}(a)+ g_a,\\[0.2em]
b' \! & \! = \! & \! r_1' \partial_{r_1}(b)+ g_b,
\end{eqnarray*}
where $g_a$, $g_b \in E[r_1,r_2,r_2']$ and $\partial_{r_1}=\partial/\partial r_1$.
This yields
$$a'b-ab' \,  = \, r_1'\big(\partial_{r_1}(a)b- a \partial_{r_1}(b) \big) + (g_a b -a g_b)  .$$
  We consider $a,b$ as elements of the differential ring $E[r_2][r_1]$ with 
derivation $\partial_{r_1}$.
Then, recognizing the coefficient of $r_1'$ as a
Wronskian determinant, we have
$$\partial_{r_1}(a)b- a \partial_{r_1}(b) \,= \, 0$$ if and only if 
$a,b$ are $E(r_2)$-linearly dependent, i.e.~$a=eb$ with $e \in E(r_2)$.
\end{proof}

\begin{lemma}
For the coefficients of $\sigma(f^+_{\lweyl,3} ) \in E\langle r_1,r_2\rangle\{(b_*)_3 \}$
the system of Lemma~\ref{lem:maxdiffideal} has for all $m >0$ no solution in $E\langle r_1,r_2\rangle$, i.e.~the 
assumptions of the lemma are satisfied.
\end{lemma}
\begin{proof}
In the expansion of $f((b_*)_3) = \sigma(f^+_{\lweyl,3})$ in the statement of Lemma~\ref{lem:maxdiffideal} only the coefficient $c_0$
contains a term with a proper derivative which is $r_1  r_2'$.

First, we prove that for every $m$ there is no solution $(\overline{y}_0,\dots,\overline{y}_{m-1}) \in E^m$. We consider the 
last equation 
of the system
whose only constant term (with respect to $y_0$, \ldots, $y_{m-1}$) is $m c_0$.
We argue that the term $r_1  r_2'$ that is thus introduced through $c_0$ into the last equation does not cancel.
In fact, since the coefficient of $r_1  r_2'$ is the constant $-1$ and $r_1$, $r_2$ are differentially  algebraically independent over $E$,
substitution of a putative solution $(\overline{y}_0,\dots,\overline{y}_{m-1}) \in E^m$
does not lead to a cancellation of $r_1  r_2'$.

Next we prove that for every $m>0$ the system has no solution $(\overline{y}_0,\dots,\overline{y}_{m-1}) $ in $E\langle r_1,r_2 \rangle^m$ 
with the property that at least one $\overline{y}_i$ satisfies $\overline{y}_i \notin E\langle r_1 \rangle (r_2)$. Suppose there is such a solution and write every $\overline{y}_i$ as a completely reduced fraction
$a_i/b_i$ with $a_i$, $b_i \in E\{r_1,r_2\}$, $b_i\neq0$. We consider then the orders of all $a_i$, $b_i$ in the differential indeterminate $r_2$
and choose an index $k$ where either $a_k$ or $b_k$ has highest order $d$. Then the assumption on the solution implies $d\geq1$.
We substitute the solution into the system and consider the $k$-th equation. We clear the denominator on the left hand side by multiplying both sides by $b_k^2$, which yields the left hand side $a_k'b_k-a_kb_k'$.
The assumption that $a_k/b_k$ is completely reduced implies 
$$\frac{a_k}{b_k} \notin E\langle r_1\rangle (r_2,\dots,r_2^{(d-1)})$$
and so Lemma~\ref{lem:technical1} applied with $K=E\langle r_1\rangle$ yields that the left hand 
side has order $d+1$ in $r_2$.
Since the maximal order in $r_2$ among the coefficients on the right hand side is one and the maximal order among the $a_i$, $b_i$ is $d\geq1$, 
we conclude that the order in $r_2$ of the right hand side is at most $d\geq1$ while the left hand side has order $d+1$, which implies that such a solution does not exist.

Proceeding in an analogous way with the roles of $r_1$ and $r_2$ exchanged, we conclude that
no solution $(\overline{y}_0,\dots,\overline{y}_{m-1})$ in $E\langle r_1,r_2 \rangle^m$ exists 
where at least one $\overline{y}_i$ satisfies $\overline{y}_i \notin E\langle r_2 \rangle (r_1)$,
where the contradiction is derived by observing that the right hand side has order at most $d \ge 1$ in $r_1$, whereas the left hand side has order $d+1$ in $r_1$.

Finally we prove that for every $m \ge 1$ there is no solution $(\overline{y}_0,\dots,\overline{y}_{m-1}) \in E(r_1,r_2)^m$ of $\Sigma_m$. Suppose that for some $m$ there is one and assume that each $\overline{y}_i = a_i/b_i$ is completely reduced.
We consider the last equation of the system.
After substitution of $(\overline{y}_0,\dots,\overline{y}_{m-1})$ into the equation we again clear the denominator on the left hand side by multiplying through by $b_{m-1}^2$.
The term $r_1r_2'$ on the right hand side does not cancel because all other terms are elements of $E[r_1,r_2]$.
Moreover, the indeterminate $r_1'$ does not occur on the right hand side, so that the left hand side  $a'_{m-1}b_{m-1}-a_{m-1}b'_{m-1}$ is independent of $r_1'$, too.
By Lemma~\ref{lem:technical2} this is the case if and only if $a_{m-1}/b_{m-1} \in E(r_2)$. This forces  the left hand side to be in $E\{ r_2 \}$, but the right hand side contains the term $r_1r_2'$.
\end{proof}

\bibliographystyle{alpha}
\bibliography{main}

\end{document}